\documentclass[12pt,reqno,palentino]{amsart}
\usepackage{mypack}
\usepackage{amscd}
\usepackage{pb-diagram}

\usepackage[backgroundcolor=white,linecolor=gray]{todonotes}

\begin{document}

\title[Morse-Conley-Floer Homology]{}

\author[T.O. Rot and R.C.A.M. Vandervorst]{}

\maketitle
\noindent {\huge {\bf   Morse-Conley-Floer Homology}} 
\vskip.8cm

\noindent 
T.O. Rot (t.o.rot@vu.nl) and R.C.A.M. Vandervorst (vdvorst@few.vu.nl)\footnote{Thomas Rot supported by NWO grant 613.001.001.} 
\vskip.3cm

\noindent{\it Department of Mathematics, VU University Amsterdam, De Boelelaan 1081a,
1081 HV Amsterdam, the Netherlands.}

\vskip1cm
\begin{sloppypar}

\noindent {\bf Abstract.}
The gradient flow of a Morse function on a smooth closed manifold generates, under suitable transversality assumptions, the Morse-Smale-Witten complex. The associated Morse homology is an invariant for the manifold, and equals the singular homology, which yields the classical Morse relations. In this paper we define Morse-Conley-Floer homology, which is an analogous homology theory for isolated invariant sets of smooth, not necessarily gradient-like, flows. We prove invariance properties of the Morse-Conley-Floer homology, and show how it gives rise to the Morse-Conley relations.\\

\noindent {\em AMS Subject Class:} 37B30, 37C10, 58E05

\noindent {\em Keywords:} Morse homology, Lyapunov functions, Conley index theory

\textwidth=12.5cm
\vskip1cm

\section{Introduction}
\label{sec:intro}
The aim of this paper is to define an analogue of Morse homology for isolated invariant sets of smooth, not necessarily gradient-like, flows. We first recall how the gradient flow of a Morse gives rise to Morse homology. 

\subsection{Morse Homology}
On a smooth, closed, $m$-dimensional manifold $M$ the \emph{Morse-Smale} pairs $(f,g)$, consisting of a smooth function $f: M\to \mR$ and a Riemannian metric $g$, are characterized by the property that
all critical points of $f$ are non-degenerate and all their stable and unstable manifolds with respect to
the negative $g$-gradient flow intersect transversally. For a  given Morse-Smale pair $(f,g)$, the \emph{Morse-Smale-Witten chain complex} consists of free abelian groups\footnote{We define the Morse homology with coefficients in $\mathbb{Z}$. Other coefficient fields are also possible. We drop the coefficient group from our notation.}  $C_k(f)$ generated by the critical points of $f$ of index $k$, and boundary operators 
$$
\partial_k(f,g,\co): C_k(f) \to C_{k-1}(f),
$$
 which count the oriented intersection between stable and unstable manifolds of the negative $g$-gradient flow of $f$. The Morse-Smale-Witten complex $(C_*(f),\partial_*(f,g,\co))$ is a chain complex and its homology is the \emph{Morse homology} $\HM_*(f,g,\co)$ of a triple $(f,g,\co)$ --- called a \emph{Morse-Smale-Witten triple} ---, where $\co$ is a choice of orientations of the unstable manifolds at the critical points of $f$. Between different triples $(f^\alpha,g^\alpha,\co^\alpha)$ and $(f^\beta,g^\beta,\co^\beta)$ there exists canonical isomorphisms $\Phi^{\beta\alpha}_*:\HM_*(f^\alpha,g^\alpha,\co^\alpha)\rightarrow \HM_*(f^\beta,g^\beta,\co^\beta)$ induced by the continuation map. This defines an inverse system. The Morse homology of the manifold is then defined by
$$
\HM_*(M) := \varprojlim \HM_*(f,g,\co),
$$
where the inverse limit is taken over all Morse-Smale-Witten triples $(f,g,\co)$, with the canonical isomorphisms. A similar construction can be carried out for non-compact manifolds, using Morse functions that satisfy a coercivity condition, cf.~\cite{Schwarz:1993wg}. 
For closed manifolds there exists an isomorphism to singular homology:
\begin{equation}
\label{eqn:hom}
\HM_*(M) \cong H_*(M;\mZ),
\end{equation}
cf.~\cite{Salamon1}, \cite{Smale:1961vr}.

The above results still apply if we consider compact manifolds with boundary, for which $df(x)\nu \not = 0$, for all $x\in \partial M$, where $\nu$ is an outward pointing normal on the boundary.\footnote{The outward pointing normal is defined as $\nu = -\frac{\nabla_g h}{|\nabla_g h|_g}$,
where $h: M \to [0,\infty)$ is smooth boundary defining function with $h^{-1}(0) = \partial M$, and $dh|_{\partial M}\not=0$.} This implies that $f$ has no critical points on the boundary.  The boundary splits as $\partial M=\partial M_-\cup \partial M_+$, where  $\partial M_-$ is the union of the  components where $df(x)\nu <0$ and $\partial M_+$  is the union of the components  where $df(x)\nu >0$. In this case the Morse homology can be linked  to the singular homology of $M$ as follows
\begin{equation}
\label{eqn:hom1}
\HM_*(M) \cong H_*(M,\partial M_-;\mZ),
\end{equation}
cf.~\cite{Kronheimer}, \cite{Schwarz:1993wg}.
When $\partial M_- = \varnothing$, i.e.~$M$ has no boundary or $df(x)\nu >0$ for all $x\in \partial M$, we have $\HM_*(M)\cong H_*(M;\mathbb{Z})$. The classical Morse relations/inequalities for Morse functions are an immediate corollary.
The isomorphism in (\ref{eqn:hom1}) also holds in the more general setting when the boundary allows points where $df(x)\nu = 0$,
with the additional requirement that such points are `external tangencies' for the negative gradient flow. The latter can also be generalized to piecewise smooth boundaries, cf. Section~\ref{sec:relblock}.

\subsection{Morse-Conley-Floer Homology} 
For arbitrary flows an analogue of Morse homology can be defined. Let $M$ be a, not necessarily compact, smooth $m$-dimensional manifold without boundary. A smooth function $\phi: \mR \times M \to M$ is called a \emph{flow}, or \emph{$\mR$-action} on $M$ if:
\begin{enumerate}
\item[(i)] $\phi(0,x) = x$, for all $x\in M$ and \item[(ii)] $\phi\bigl(t,\phi(s,x)\bigr) = \phi(t+s,x)$, for all $s,t \in \mR$ and $x\in M$.
\end{enumerate}
A smooth flow satisfies the differential equation
$$
\frac{d}{dt} \phi(t,x) = X\bigl(\phi(t,x)\bigr),\quad\hbox{with}\quad X(x) = \frac{d}{dt} \phi(t,x)\Bigl|_{t=0} \in T_xM,
$$
the associated vector field of $\phi$ on $M$. A subset $S\subset M$ is called \emph{invariant} for $\phi$ if $\phi(t,S) = S$, for all $t\in \mR$. Examples of invariant sets are fixed points and periodic orbits.

A compact neighborhood $N\subset M$ is called an \emph{isolating neighborhood} for $\phi$ if $\Inv(N,\phi) \subset \Int(N)$, where $$
\Inv\bigr(N,\phi\bigl) = \{x\in N~|~\phi(t,x) \in N,~\forall t\in \mR\},
$$
is called the maximal invariant set in $N$. An invariant set $S$ for which there exists an isolating neighborhood $N$  with $S=\Inv(N,\phi)$, is called an \emph{isolated invariant set}. Note that there are many isolating neighborhoods for an isolated invariant set $S$. Isolated invariant sets are compact. 
For analytical reasons, we need isolating neighborhoods with an appropriate manifold structure, in the sense that
boundary of such a neighborhood is piecewise smooth and the flow $\phi$ is transverse to the smooth components of the boundary. Such isolating neighborhoods are called \emph{isolating blocks}, cf.~Definition~\ref{defn:block}. 
Every isolated invariant set $S = \Inv(N,\phi)$ admits an isolating block $B\subset N$.  
Isolating blocks are used to prove the existence of Lyapunov functions.

A smooth \emph{Lyapunov function} for an isolated invariant set $S$ is a smooth function $\f : M\rightarrow \mR$, such that $\frac{d}{dt}\bigr|_{t=0} \f(\phi(t,x))<0$ for $x\in N\setminus S$. Denote the set of Lyapunov functions by $\lyap(S,\phi)$. This set is non-empty, cf.~Proposition~\ref{prop:lyap1}. If $(f,g)$ is a Morse-Smale pair, with $f$ is an arbitrary small Morse perturbation of a Lyapunov function $\f$, then one can define the Morse homology for the quadruple $(f,g,N,\co)$, for some choice of orientation $\co$ of unstable manifolds of the critical points of $f$ in $N$. The Morse homology $\HM_*(f,g,N,\co)$ is independent (up to canonical isomorphisms) of the isolating block, the Lyapunov function, the Morse perturbation, the metric and the chosen orientations, which leads to the definition of the \emph{Morse-Conley-Floer homology} of $(S,\phi)$ as an inverse limit
\begin{equation*}
\HI_*(S,\phi) := \varprojlim \HM_*(f,g,N,\co).
\end{equation*}
This is also an invariant for the pair $(N,\phi)$, for any isolating neighborhood $N$ for $S$, if one takes the inverse limit over a fixed isolating neighborhood $N$.                                                                      
The important properties of the Morse-Conley-Floer homology  can be summarized as follows:

\begin{itemize}
\item[(i)] 
Morse-Conley-Floer homology $\HI_k(S,\phi)$ is of finite rank for all $k$ and $\HI_k(S,\phi) =0$, for all $k<0$ and $k> \dim M$.
\item[(ii)] 
If $S=\varnothing$ for some isolating neighborhood $N$, i.e. $\Inv(N,\phi) = \varnothing$, then $\HI_*(S,\phi) \cong 0$. Thus  $\HI_*(S,\phi) \not = 0$ implies that $S \not = \varnothing$, which is an important tool for finding non-trivial isolated invariant sets.
\item[(iii)] 
The Morse-Conley-Floer homology satisfies a global continuation principle. If isolated invariant sets $(S_0,\phi_0)$ and $(S_1,\phi_1)$ are related by continuation, see Definition~\ref{def:globalcont}, then
\begin{equation}
\label{eqn:cont-S}
\HI_*(S_0,\phi_0)\cong \HI_*(S_1,\phi_1).
\end{equation}
This allows for the computation of the Morse-Conley-Floer homology in non-trivial examples.
\item[(iv)] 
Let $\{S_i\}_{i\in I}$, indexed by a finite poset $(I,\le)$, be a Morse decomposition for $S$, see Definition \ref{defn:MD}. The sets $S_i$ are Morse sets and are isolated invariant sets by definition. Then,
\begin{equation}
\label{eqn:MRfor-S}
\sum_{i\in I} P_t(S_i,\phi)  = P_t(S,\phi) +  (1+t) Q_t.
\end{equation}
where $P_t(S,\phi)$ is the Poincar\'e polynomial of $\HI_*(S,\phi)$, and $Q_t$ is a polynomial with non-negative coefficients. These relations are called the Morse-Conley relations and generalize the classical Morse relations for gradient flows. 
\item[(v)] 
Let $S$ be an isolated invariant set for $\phi$ and let $B$ be an isolating block for $S$, see Definition\ \ref{defn:block}. Then
\begin{equation}
\HI_*(S,\phi) \cong H_*(B,B_-;\mZ),
\end{equation}
where 
$B_-  = \{ x\in \partial B~|~X(x)~ \hbox{is outward pointing}\}$\footnote{A vector $X(x)$ is outward pointing at a point $x\in \partial B$ if
$X(x) h<0$, where the function $h: B\to [0,\infty)$ is any boundary defining function for $B_-$.
An equivalent characterization is $g(X(x),\nu(x))>0$, where $\nu$ is the outward pointing $g$-normal vector field on $B_-$. These conditions do not depend on $h$ nor $g$.} 
and is called the `exit set'.
\end{itemize}

Note that in the case that $\phi$ is the gradient flow of a Morse function on a compact manifold, then Property (v) recovers the results of Morse homology, by setting $S=M$. Property (v) also justifies the terminology Morse-Conley-Floer homology, since the construction uses Morse/Floer homology and recovers the classical homological Conley index. In the subsequent sections we construct the Morse-Conley-Floer homology and prove the above properties.

\subsection{Functoriality}

We recall the functorial behavior of Morse homology. Let $(f^\alpha,g^\alpha,\co^\alpha)$ and $(f^\beta,g^\beta,\co^\beta)$ be Morse-Smale-Witten triples on closed manifolds $M^\alpha$ and $M^\beta$. Under suitable transversality assumptions, cf.\ \cite{AS-1}, \cite{Kronheimer}, a smooth function $h^{\beta\alpha}:M^\alpha\rightarrow M^\beta$ induces a map
$$
h^{\beta\alpha}_k: C_k(f^\alpha)\rightarrow C_k(f^\beta),
$$
by counting intersections points in $W^u(x)\cap (h^{\beta\alpha})^{-1}(W^s(y))$, where $x$ and $y$ are critical points of index $k$ of $f^\alpha$ and $f^\beta$. The induced map is a chain map which commutes with the canonical isomorphisms, and hence induces a map $h^{\beta\alpha}_*:HM_*(M^\alpha)\rightarrow HM_*(M^\beta)$ between the Morse homologies. On the homology level the induced map is functorial.

For Morse-Conley-Floer homology one expects a similar statement. Let $\phi^\alpha$ and $\phi^\beta$ be flows on $M^\alpha$ and $M^\beta$ respectively, and assume that $h^{\beta\alpha}:M^\alpha\rightarrow M^\beta$ is a flow map, i.e.~it is equivariant with respect to the flows $\phi^\alpha$ and $\phi^\beta$. Let $N^\beta$ be an isolating neighborhood for the isolated invariant set $S^\beta$ of the flow $\phi^\beta$, then it follows that $N^\alpha=(h^{\beta\alpha})^{-1}( N^\beta)$ is an isolating neighborhood for $S^\alpha=(h^{\beta\alpha})^{-1}(S^\beta)$ of $\phi^\alpha$. As before we expect an induced map
$$
h^{\beta\alpha}_*:HI_*(S^\alpha,\phi^\alpha)\rightarrow HI_*(S^\beta,\phi^\beta),
$$
by counting intersections which behaves functorially. The details of this will be taken up in a sequel, and we expect we can apply the functoriality to prove a fixed point theorem similar to the fixed point theorem established in~\cite{McCord}.

\subsection{Generalizations}

Morse homological tools were used in \cite{Jiang} to study two types of degenerate gradient systems: Morse-Bott functions, as well as functions with degenerate but isolated critical points. Morse-Conley-Floer homology also gives these results, but more types of degeneracies are admissible. 

Another important point to make is that the ideas explained in this paper will be exploited in order to  develop a Conley type index theory for flows on infinite dimensional spaces as well as indefinite elliptic equations.  In~\cite{Reineck} Lyapunov functions are used to construct gradient vector fields that allow computation of the Conley Index through continuation. Our approach is to avoid the Conley index, but to emphasize the intrinsically defined index, which is better behaved infinite dimensional settings. The objective is to generalize to arbitrary flows in infinite dimensions as well as extensions of Floer homology beyond gradient systems. Morse homology is an intrinsically defined invariant for $M$, which is isomorphic to the homology of $M$, and analogously Morse-Conley-Floer homology is an intrinsically defined invariant for isolated invariant sets of flows, which is in the finite dimensional case isomorphic to the Conley index. Especially with respect to our long term goal of developing Conley index type invariants for strongly indefinite flows, the Morse-Conley-Floer homology  approach can be used to define such invariants, by establishing appropriate Lyapunov functions. 

In a series of papers M. Izydorek \cite{Iz-1,Iz-2} developed a Conley index approach for strongly indefinite flows based
on Galerkin approximations and a novel cohomology theory based on spectra. Using the analogues of index pairs he established a cohomological Conley for strongly indefinite systems. 

This paper develops an  intrinsic approach towards Conley index in the finite dimensional setting.
An important application is to adopt this approach in order to  develop an intrinsic   Conley index
for strongly indefinite flows, but more importantly to strongly indefinite elliptic problems. In the latter case an appropriate flow on an ambient space does not exist and index pairs can not be defined, cf.\ \cite{MV}, which makes it unsuitable for the approach in \cite{Iz-1,Iz-2}.
This infinite dimensional approach using Lyapunov functions will be the subject of further research.

\section{Isolating blocks and Lyapunov functions}
\label{sec:isol-block}
In this section we discuss the existence of isolating blocks and Lyapunov functions.
\begin{figure}
\def\svgwidth{.7\textwidth}
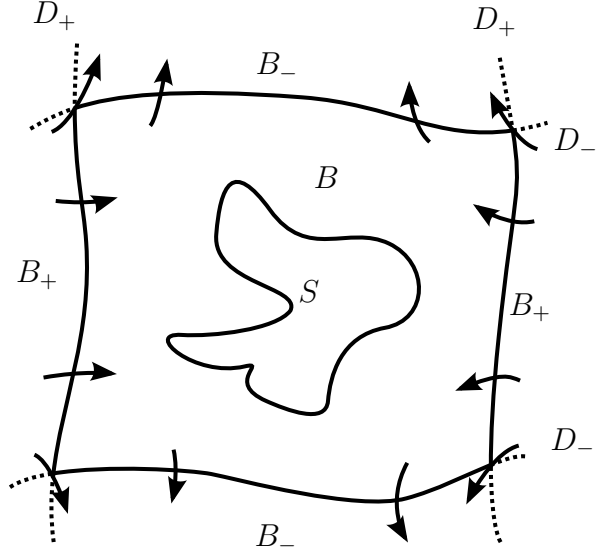
\caption{An isolating block $B$ for an isolated invariant set $S$. The boundary of $B$ decomposes into $\partial B=B_+\cup B_-\cup B_\pm$. Here $B_\pm=B_+\cap B_-$ consists of the four corner points.}
\label{fig:block}
\end{figure}
\subsection{Isolating blocks}
\label{subsec:isbl}
Isolated invariant sets admit isolating neighborhoods with piecewise smooth boundaries known as isolating blocks.
\begin{definition}
\label{defn:block}
{\em
An isolating neighborhood $B\subset M$  for $\phi$ is called a
\emph{smooth isolating block} if $B$ is a compact $m$-dimensional submanifold with piecewise smooth boundary $\partial B$ and the boundary satisfies the following requirements:
\begin{enumerate}
\item [(i)] the boundary decomposes as $\partial B = B_{+}\cup B_{-}\cup B_\pm$, with $B_{+}\cap B_{-} = B_\pm$ and
$B_{-}\setminus B_\pm$, $B_{+}\setminus B_\pm$ (when non-empty) are smooth $(m-1)$-dimensional submanifolds of $M$;

\item [(ii)] there exist open smooth $(m-1)$-dimensional submanifolds  $D_{-}, D_{+}$  such that
$B_{+}\subset D_{+}$, $B_-\subset D_-$ and $D_{-}\cap D_{+} = B_\pm $ is a $(m-2)$-dimensional submanifold (when non-empty);
\item [(iii)] The flow is transverse to $D_\pm$, i.e.~$\phi \pitchfork D_{\pm}$, and for any $x\in D_\pm$ there exists an $\epsilon>0$ such that
$\phi(I_\epsilon^\pm,x) \cap B = \varnothing$, where $I^-_\epsilon = (0,\epsilon)$ and  $I^+_\epsilon = (-\epsilon,0)$.
\end{enumerate}
The sets $B_{-}\setminus B_\pm$ and $B_{+}\setminus B_\pm$ are also called \emph{egress} and \emph{ingress} respectively and are characterized by
the property that $X\cdot \nu >0$ on $B_{-}\setminus B_\pm$ and $X\cdot \nu <0$ on $B_{+}\setminus B_\pm$, where $\nu$ is the outward pointing $g$-normal vector field on $\partial B\setminus B_{\pm}$.
}
\end{definition}
\begin{remark}
{\em
In \cite{Wilson:1973vx}, Wilson and Yorke call this concept an isolating block \emph{with corners}. For the sake of brevity we will refer to such isolating neighborhoods as (smooth) isolating blocks.
}
\end{remark}

All isolated invariant sets admit isolating blocks. 

\begin{proposition}[Wilson-Yorke \cite{Wilson:1973vx}]
\label{prop:WY}
For any isolating neighborhood $N\subset M$ of $S$, there exists a smooth isolating block $B\subset N$, such that $\Inv(B,\phi) =  \Inv(N,\phi) =  S$.
\end{proposition}

\subsection{Lyapunov functions}
\label{subsec:lyapfnc}
The existence of isolating blocks implies the existence of global Lyapunov functions with special properties with respect to isolated invariant sets.

\begin{definition}
\label{defn:lyap}
{\em
A smooth  \emph{Lyapunov function} for $(S,\phi)$ is a smooth function $\f: M \to \mR$  
satisfying the properties:
\begin{enumerate}
\item [(i)] $\f\bigl|_{S} = {\rm constant}$;
\item [(ii)] $\frac{d}{dt}\bigr|_{t=0} f\bigl(\phi(t,x)\bigr) <0$ for all
$x\in N\setminus S$, for some isolating neighborhood $N$ for $S$.\footnote{Property (ii) will also be referred to as the \emph{Lyapunov property} with respect to an isolating neighborhood $N$ for $S$.}
\end{enumerate}
The set of smooth Lyapunov functions for $S$ is denoted by $\lyap(S,\phi)$.
}
\end{definition}

In the subsequent arguments we will always assume that the gradient of $f$ defines a global flow. All arguments in the subsequent sections only care about the behavior of $f$ inside $N$. Without loss of generality we can take $f$ to be constant outside a neighborhood of $N$. This does not affect the Lyapunov property of $f$, and it ensures that any gradient of $f$ defines a global flow. If $M$ is closed the gradient of $f$ always defines a global flow.

The set of smooth Lyapunov functions is non-empty, cf. Proposition~\ref{prop:lyap1}, and also convex in the following sense. 

\begin{lemma}
\label{prop:convex}
Let $f_\phi^\alpha, f_\phi^\beta: M\to \mR$  be Lyapunov functions for $S$. Then,
\begin{enumerate}
\item[(i)] for all $\lambda\in \mR$ and all $\mu\in \mR^+$, it holds that $\lambda + \mu f_\phi^\alpha \in \lyap(S,\phi)$;
\item[(ii)] for all $\lambda,\mu \in \mR^+$, it holds that $f_\phi^\gamma=\lambda f_\phi^\alpha+\mu f_\phi^\beta\in \lyap(S,\phi)$.

\end{enumerate}

\end{lemma}
\begin{proof}
The first property is immediate. For the second property we observe that $N^\gamma = N^\alpha\cap N^\beta$
 is an isolating neighborhood for $S$ and
$$\frac{d}{dt}\biggr|_{t=0}f_\phi^\gamma\bigl(\phi(t,x)\bigr) = 
\lambda \frac{d}{dt}\biggr|_{t=0}f_\phi^\alpha \bigl(\phi(t,x)\bigr)  +\mu \frac{d}{dt}\biggr|_{t=0}f_\phi^\beta\bigl(\phi(t,x)\bigr) <0,
$$
for all $x\in N^\gamma\setminus S$.
\end{proof}

The next result is an adaption of Lyapunov functions  to smooth Lyapunov functions as described in \cite{Robbin:1992wp} and \cite{Wilson:1973vx}, which as a consequence shows that the set $\lyap(S,\phi)$ is non-empty. The following proposition is a modification of a result in \cite{Robbin:1992wp}.

\begin{figure}
\def\svgwidth{.4\textwidth}
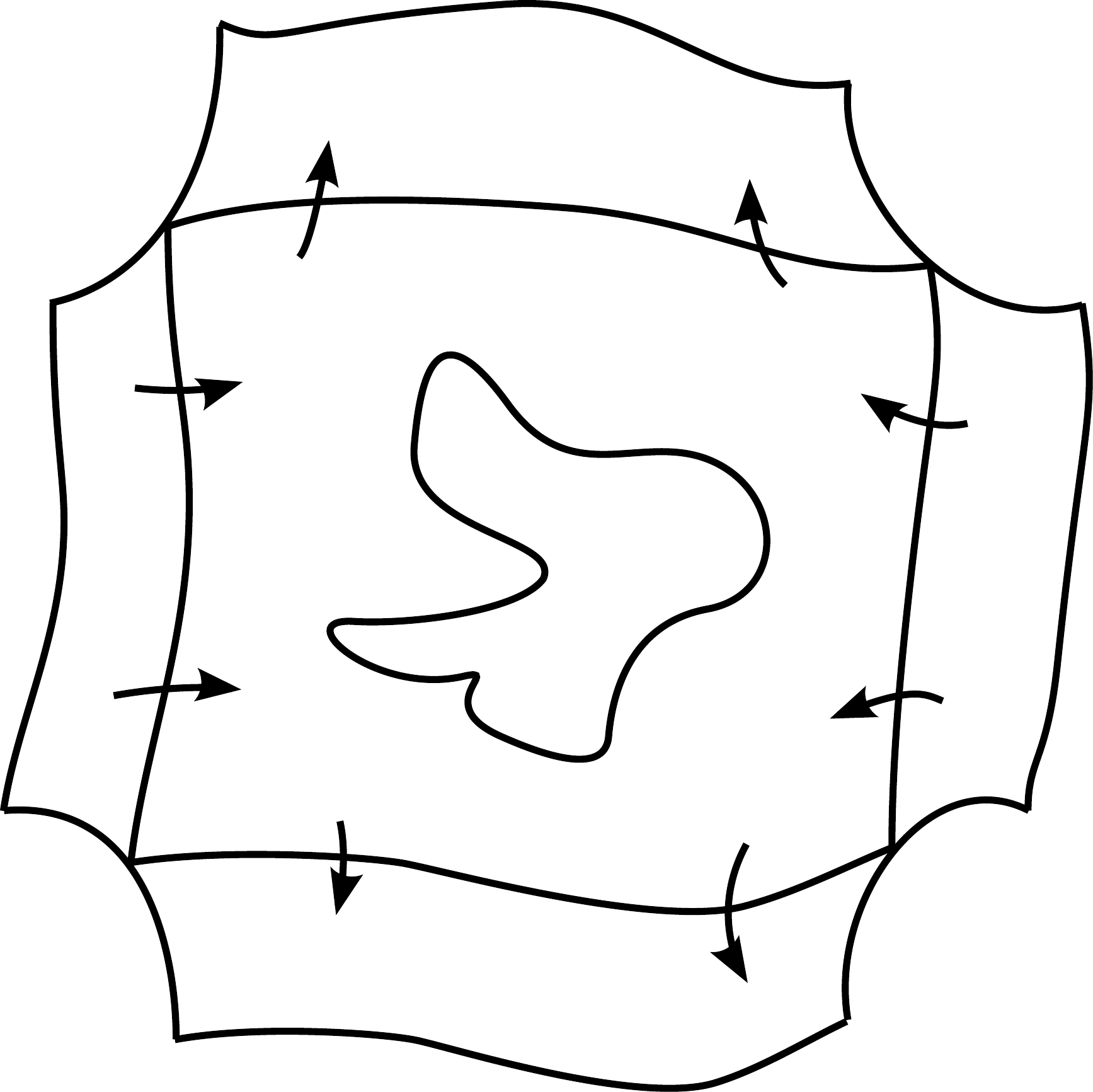
\def\svgwidth{.4\textwidth}
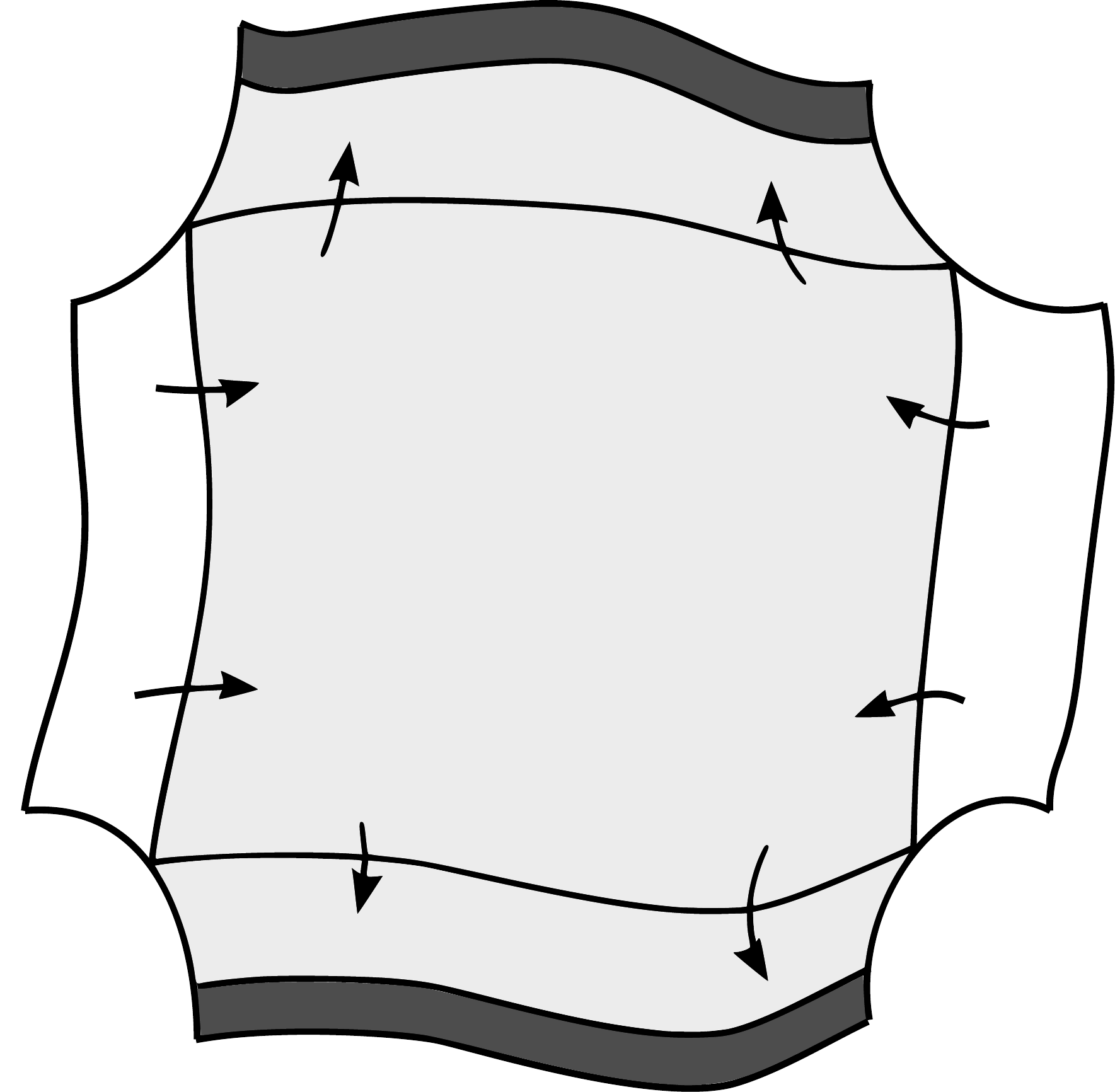
\caption{The isolating block $B$ of $S$ is slightly enlarged using the flow $\phi$ to the manifold with piecewise smooth boundary $B^\dagger$.}
\label{fig:bdagger}
\end{figure}
\begin{proposition}
\label{prop:lyap1}
Let $S\subset M$ be an isolated invariant set for $\phi$ and $B\subset M$ a smooth isolating block for $S$.
Then, there exists a smooth Lyapunov function $\f: M \to \mR$ for $S$, such that the Lyapunov property of $\f$ holds with respect to $B$. 
\end{proposition}
\begin{proof}

Consider the following manifold $B^{\dagger}$ with piecewise smooth boundary, defined by   
$$
B^{\dagger} = \phi\bigl([-3\tau,3\tau],B\bigr) \subset M,
$$ 
for $\tau>0$ sufficiently small. By construction $\Inv(B^{\dagger},\phi) = S$, cf.\ Figure \ref{fig:bdagger}. Choose a smooth cut-off function\footnote{See   \cite{Wilson:1973vx}, Thm. 1.8, for a proof of the existence.} $\zeta: B^{\dagger} \to [0,1]$
such that 
$$
\zeta^{-1}(1)=B\quad {\rm and}\quad\zeta^{-1}(0)=\phi([2\tau,3\tau],B_{-})\cup\phi([-3\tau,-2\tau],B_{+})
$$ 
and define the vector field $X^{\dagger} = \zeta X$ on $B^{\dagger}$.
The flow generated by the vector field $X^{\dagger}$ is denoted by $\phi^{\dagger}$ and $\phi^{\dagger}|_{B} = \phi$.

For the flow $\phi^{\dagger}:\mR\times B^{\dagger} \to B^{\dagger}$ we  identify the following attracting neighborhoods\footnote{A compact neighborhood $N\subset M$ is attracting if $\phi(t,N) \subset \Int(N)$ for all $t>0$. Similarly, a compact neighborhood is repelling if 
$\phi(t,N) \subset \Int(N)$ for all $t<0$.}
$$
U = \phi([0,3\tau],B)\quad{\rm and}\quad V  =  \phi([\tau,3\tau],B_-),
$$
for the flow $\phi^\dagger$ on $B^\dagger$. This yields the attractors $A_U = \omega(U)$ and $A_V = \omega(V)$ for $\phi^\dagger$ in $B^\dagger$. Following the theory of attractor-repeller pairs\footnote{A set $A\subset M$ is an \emph{attractor} for a flow $\phi$ if there exits a neighborhood $U$ of $A$ such that $A=\omega(U)$.
The associated \emph{dual repeller} is given by $A^* = \alpha(M\setminus U)$. The pair $(A,A^*)$ is called an \emph{attractor-repeller pair} for $\phi$. The pairs $(\varnothing,M)$ and $(M,\varnothing)$ are trivial attractor-repeller pairs.} (cf. \cite{KMV,KMV2}, \cite{Robbin:1992wp})
the dual repellers in $B^\dagger$ are given by $A^*_U = \alpha(W)$ and $A^*_V= \alpha(Z)$, where $W$ and $Z$ are given by
$$
W=B^\dagger\setminus U=\phi([-3\tau,0),B_+)\quad{\rm and}\quad Z=B^\dagger\setminus V=\phi([-3\tau,\tau),B).$$
By the theory of Morse decompositions (cf. \cite{KMV}, \cite{Robbin:1992wp}) we have
$$
S = A_U \cap A^*_V,
$$
cf.\ Figure\ \ref{fig:bdaggerauv}. By Proposition 1.4 in \cite{Robbin:1992wp} we obtain Lyapunov functions $f_{A_U}$ and $f_{A_V}$ which are smooth
on $B^\dagger \setminus (A_U\cup A^*_U)$ and $B^\dagger \setminus (A_V\cup A^*_V)$ respectively. They satisfy $\frac{d}{dt}f_{A_U}\bigl(\phi(t,x)\bigr) <0$ for all
$x\in B^\dagger\setminus (A_U\cup A_U^*)$, and $f_{A_U}^{-1}(0) = A_U$, and $f_{A_U}^{-1}(1) = A^*_U$, and the same for $f_{A_V}$.

Define
$$
\f = \lambda f_{A_U} +\mu  f_{A_V},\quad \lambda,\mu>0.
$$
Clearly,
 $\f\bigl|_{S}= {\rm constant}$ and $\frac{d}{dt}\f\bigl(\phi(t,x)\bigr) <0$ for all
$x\in B \setminus S$ and all times $t>0$. We extend  $\f$ to a smooth function on $M$.
\end{proof}
\begin{figure}
\def\svgwidth{.7\textwidth}
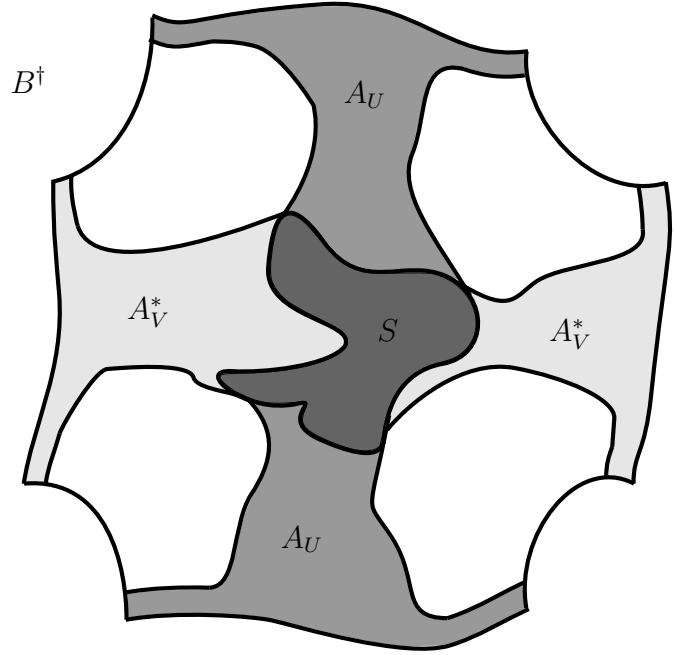
\caption{An illustration of the fact that $S=A_U\cap A^*_V$.}
\label{fig:bdaggerauv}
\end{figure}

\section{Gradient flows, Morse functions and Morse-Smale flows}
\label{sec:grad}

The gradient flow of a Lyapunov function $\f$ for $(S,\phi)$ gives information about the set $S$. 
\subsection{Gradient flows}
\label{subset:grad}
Let $g$ be a Riemannian metric on $M$ and consider the (negative) $g$-gradient flow equation
$
x' = - \nabla_g \f(x)
$,
where $\nabla_g \f$ is the gradient of $\f$ with respect to the Riemannian metric $g$.
The differential equation generates a global flow $\psi_{(\f,g)}:\mR\times M\to M$ by assumption. We say that $\psi_{(\f,g)}$ is the gradient flow of the pair $(\f,g)$ for short. The set of critical points of $\f$ is denoted by 
$$
\crit(\f) = \{x\in M~|~d\f(x) = 0\}.
$$

\begin{lemma}
\label{lem:crit1}
$\crit(\f) \cap N \subset S$.
\end{lemma}

\begin{proof}
Because $\f$ is a Lyapunov function for $\phi$ we have 
$$
 d\f(x) X(x) = \frac{d}{dt}\biggr|_{t=0}\f\bigl(\phi(t,x)\bigr) <0,
$$
for all $x\in N\setminus S$. This implies that $d\f(x) \not = 0$ for $x\in N\setminus S$ and thus $\crit(\f) \cap  N \subset S$.  
\end{proof}
\begin{remark}
{\em
The reversed inclusion is not true in general. The flow $\phi$ on $M=\mR$, generated by the equation $x'=\frac{x^2}{1+x^2}$, has an equilibrium at $0$ and $S=\{0\}$ is an isolated invariant set. The function $x\mapsto -x$ is a Lyapunov function for this flow without any critical points.
}
\end{remark}
By construction $N$ is an isolating neighborhood for $\phi$, but it is also an isolating neighborhood for the gradient flow $\psi_{(\f,g)}$.

\begin{lemma}
\label{lem:inv2}
For any choice of metric $g$, the set $N\subset M$ is an isolating neighborhood for $\psi_{(\f,g)}$ and $S_{(\f,g)} = \Inv(N,\psi_{(\f,g)}) = \crit(\f)\cap N\subset S$. 
\end{lemma}

\begin{proof}
Let $x \in \Inv(N,\psi_{(\f,g)})$ and let $\gamma_{x} = \psi_{(\f,g)}(\mR,x)$ be the orbit through $x$, which is bounded since $\gamma_{x}\subset N$, and $N$ is compact. For the gradient flow $\psi_{(\f,g)}$ the following identity holds:
\begin{equation}
\label{eqn:grad-id}
\frac{d}{dt}\f\bigl(\psi_{(\f,g)}(t,x)\bigr) = -| \nabla_g \f\bigl(\psi_{(\f,g)}(t,x)\bigr) |_g^2\le 0,
\end{equation}
and  the function $t\mapsto \f\bigl(\psi_{(\f,g)}(t,x)\bigr)$ is decreasing and bounded. The latter implies that $\f\bigl(\psi_{(\f,g)}(t,x)\bigr)\to c_\pm$ as $t\to \pm \infty$. For any point $y\in\omega(x)$ there exist times $t_n\to\infty$ such that $\lim_{n\to\infty}\psi_{(\f,g)}(t_n,x)=y$, and hence 
$$
\f(y)= \f\bigl(\lim_{n\to \infty} \psi_{(\f,g)}(t_n,x)\bigr) = \lim_{n\to\infty} \f(\psi_{(\f,g)}(t_n,x)) = c_+.
$$
Suppose $y\not \in \crit(\f)\cap N$, then for any $\tau>0$, we have $\f\bigl(\psi_{(\f,g)}(\tau,y)\bigr)<c_+$,
since $d\f(y)\not = 0$ and thus $\frac{d}{dt}\f\bigl(\psi_{(\f,g)}(t,y)\bigr)\bigr|_{t=0}<0$.
On the other hand, by continuity and the group property, $\psi_{(\f,g)}(t_n+\tau,x)\to\psi_{(\f,g)}(\tau,y)$, which implies 
$$
c_+ = \lim_{n\to \infty} \f\bigl(\psi_{(\f,g)}(t_n+\tau,x_0)\bigr) = \f\bigl(\psi_{(\f,g)}(\tau,y)\bigr)<c_+,
$$
a contradiction. Thus the limit points are contained in $\crit(\f)\cap N$. The same argument holds for points $y\in \alpha(x)$. This proves that $\omega(x), \alpha(x) \subset \crit(\f)\cap N\subset S$ for all $x\in \Inv(N,\psi_{(\f,g)})$.
Since $\f$ is constant on $S$ it follows that $c_+=c_-=c$ and thus $\f\bigl(\psi_{(\f,g)}(t,x)\bigr) =c$ for all $t\in \mR$, i.e.
$\f\bigl|_{\gamma_x}  = c$.
Suppose $x\not \in \crit(\f)\cap N$, then $d\f(x) \not = 0$.
Using Equation \bref{eqn:grad-id} at $t=0$ yields a contradiction due to the fact that $\f$ is constant along $\gamma_{x}$.
Therefore $\Inv(N,\psi_{(\f,g)}) \subset  \crit(\f)\cap N \subset S$. This completes the proof since $\crit(\f)\cap N \subset \Inv(N,\psi_{(\f,g)})$.
\end{proof}

\subsection{Morse functions}
\label{subsec:Morsefnc}
A smooth function $f: M\to \mR$ is called \emph{Morse} on $N$ if $\crit(f)\cap N \subset \Int(N)$ and $\crit(f)\cap N$ consists of only non-degenerate critical points.   

Then the local 
stable and unstable manifolds of a critical point $x \in \crit(f)\cap N$, with respect to the gradient flow $\psi_{(f,g)}$, defined by $x' = -\nabla_g f(x)$, are given by
\begin{align*}
W^s_{\rm loc}(x;N)&=\{z\in M\,|\, \psi_{(f,g)}(t,z) \in N, \forall t\ge 0,~\lim_{t\rightarrow\infty} \psi_{(f,g)}(t,z)=x\},\\
W^u_{\rm loc}(x;N)&=\{z\in M\,|\, \psi_{(f,g)}(t,z) \in N, \forall t\le 0,~\lim_{t\rightarrow-\infty} \psi_{(f,g)}(t,z)=x\}.
\end{align*}
The sets $W^s(x)$ and $W^u(x)$ are the global stable and unstable manifolds respectively, which are defined without the restriction of points in $N$. We will write $S_{(f,g)}=\Inv(N, \psi_{(f,g)})$ for the maximal invariant set of $\psi_{(f,g)}$ inside $N$.

\begin{lemma}
\label{prop:invariantset}
Suppose $f$ is Morse on $N$, then
the maximal invariant set in $N$ of the gradient flow $\psi_{(f,g)}$ is characterized by
\begin{equation}
S_{(f,g)}=\bigcup_{x,y\in \crit(f)\cap N} W^u_{\rm loc}(x;N) \cap W^s_{\rm loc}(y;N).
\end{equation}
\end{lemma}

\begin{proof}
If $z\in W^u_{\rm loc}(x;N) \cap W^s_{\rm loc}(y;N)$ for some $x,y \in \crit(f)\cap N$, then by definition, $\psi_{(f,g)}(t,z)\in N$ for all $t\in \mR$, and hence $z\in  S_{(f,g)}$ which shows that $W^u_{\rm loc}(x;N) \cap W^s_{\rm loc}(y;N) \subset S_{(f,g)}$. Conversely, if $z\in S_{(f,g)}$, then $\psi_{(f,g)}(t,z)\in S_{(f,g)} \subset N$ for all $t\in \mR$, i.e. $\gamma_z \subset N$. By the the same arguments as in  the proof of Lemma~\ref{lem:inv2} the limits $\lim_{t\rightarrow \pm\infty} \psi_t(z)$ exist and are critical points of $f$ contained in $S_{(f,g)}$ (compactness of $S_{(f,g)}$).   Therefore,  $z\in W^u_{\rm loc}(x;N) \cap W^s_{\rm loc}(y;N)$, for some $x,y\in\crit(f)\cap N$ which implies that $S_{(f,g)} \subset \bigcup_{x,y\in \crit(f)\cap N}W^u_{\rm loc}(x;N) \cap W^s_{\rm loc}(y;N))$.
\end{proof}

\subsection{Morse-Smale flows and isolated homotopies}
Additional structure on the set of connecting orbits is achieved by the Morse-Smale property.
\begin{definition}
\label{defn:Morse-Smale}
{\em
A gradient flow $\psi_{(f,g)}$
is called \emph{Morse-Smale} on $N$ if 
\begin{enumerate}
\item[(i)] $S_{(f,g)} \subset \Int(N)$;
\item[(ii)] $f$ is Morse on $N$, i.e.~critical points in $N$ are non-degenerate;
\item[(iii)] $W^{u}_{\rm loc}(x;N) \pitchfork W^{s}_{\rm loc}(y;N)$\footnote{The symbol $\pitchfork$ indicates that the intersection is transverse in the following sense. For each $p\in W^{u}_{\rm loc}(x;N) \cap W^{s}_{\rm loc}(y;N)$, we have that $T_p W^u(x)+T_pW^s(y)=T_p M$.}, for all $x,y\in \crit(f)\cap N$. 
\end{enumerate}
In this setting the pair $(f,g)$ is called a \emph{Morse-Smale pair} on $N$.
 }
\end{definition}

\begin{remark}
\label{rmk:notsuited}
{\em
In general Lyapunov functions are not Morse. We want to perturb the Lyapunov function to obtain a Morse-Smale pair, which we will use to define invariants. Not all Morse-Smale pairs are suitable for the construction of invariants as the following example shows. Let $\phi(t,x) = xe^t$ be a flow on $M= \mR$ and let $N=[-1,1]$ be a isolating block for $\phi$, with $S= \Inv(N,\phi) = \{0\}$.
The function $\f(x) = -\frac{1}{4} x^4$ is a Lyapunov for $(S,\phi)$. 
Let $g$ be the standard inner product on $\mR$, then the function $f(x) = \frac{1}{2} x^2$
yields a Morse-Smale flow $\psi_{(f,g)}(t,x) = xe^{-t}$ via $x' = - f'(x) = -x$ and thus $(f,g)$  is a Morse-Smale pair on $N$.
The flow $\psi_{(f,g)}$ obviously displays the wrong dynamical behavior. The reason is that one cannot find a homotopy between
$\f$ and $f$ which preserves isolation with respect to $N$. This motivates the following definition.
}
\end{remark}

\begin{definition}
\label{defn:isolMS}
{\em
Let $\psi_{(h,e)}$ be the gradient flow of $h$ with respect to the metric $e$, with the property that
 $S_{(h,e)} \subset \Int(N)$. A Morse-Smale pair $(f,g)$ on $N$ is \emph{isolated homotopic} to $(h,e)$ if
there exists a smooth homotopy $(f_\lambda,g_\lambda)_{\lambda \in [0,1]}$ between $(f,g)$ and $(h,e)$ such that 
$S_{(f^\lambda,g^\lambda)}\subset \Int(N)$ for all $\lambda \in [0,1]$.\footnote{The flows $\psi_{(f^\lambda,g^\lambda)}$ are generated by
the equations $x' = - \nabla_{g^\lambda} f^\lambda(x)$.}
The set  of such Morse-Smale pairs is denoted by
$\IMS(h,e;N)$.
}
\end{definition}

Isolation is well behaved under perturbation, cf.~\cite{Conley}.
\begin{proposition}
The set of flows preserving isolation of $N$ is open in the compact-open topology.
\label{prop:isolation}
\end{proposition}
\begin{proof}
Let $\phi$ be a flow with isolating neighborhood $N$. Isolation implies that for all $p\in \partial N$ there exists a $t\in\mR$ such that $\phi(t,p)\in M\setminus N$. By continuity there exist compact neighborhoods $U_p\ni p$ and $I_t\ni t$ such that $\phi(I_t,U_p)\in M\setminus N$. The compactness of $\partial N$ implies that a finite number of neighborhoods $U_{p_i}$ cover $\partial N$. Now we have that
$$
\phi\in \bigcap_i \{\psi\in C(\mR\times M, M)\,|\, \psi(I_{t_i},U_{p_i})\subset M\setminus N\},
$$
and all flows in this open set are isolating. 
\end{proof}

The following proposition shows that the set of isolated homotopies $\IMS(h,e;N)$ is not empty.

\begin{proposition}
\label{prop:morse}
Let $\psi_{(h,e)}$ be a gradient flow of $(h,e)$, with the property that
 $S_{(h,e)}  \subset \Int(N)$. Then, for each $f$ sufficiently $C^2$-close to $f$, such that $(f,e)$ is a 
Morse-Smale pair on $N$, then $(f,e)$ is isolated homotopic to $(h,e)$. Consequently, $\IMS(h,e;N) \cap \{(f,e)~|~\Vert f-h\Vert_{C^2}<\epsilon\} \not = \varnothing$ for every $\epsilon >0$.
\end{proposition}

\begin{proof}
The existence of a smooth functions $f$ for which  the Morse-Smale property holds for $-\nabla_g f$, with $g=e$, follows from the results in \cite{Austin:1995tb}.
Since $N$ is an isolating neighborhood for $\psi_{(h,e)}$, a sufficiently $C^2$-small perturbation $f$ of $h$ implies that $\nabla_{g} f$ is $C^1$-close
to the vector field $\nabla_e h$. Now $C^1$-vector fields define (local) flows, and the set of flows preserves isolation is open, by Proposition~\ref{prop:isolation}. Consequently, $S_{(f,g)} = \Inv(N,\psi_{(f,g)}) \subset \Int(N)$. If the $C^2$-perturbation is small enough there is a path connecting $(f,g)$ and $(h,e)$ preserving isolation, which proves that $(f,g) \in \IMS(h,e;N)$.
\end{proof}

\begin{remark}
\label{rmk:baire}
{\em
In \cite{Austin:1995tb} it is proved that the set of smooth functions 
on a compact manifold is a Baire space. This can be used to strengthen the statement in Proposition \ref{prop:morse} as follows:
the set of smooth functions 
in a sufficiently small $C^2$-neighborhood of 
a Morse function on $N$ is 
a  Baire space.
}
\end{remark}

\begin{remark}
\label{rmk:approx}
{\em
In \cite{Smale:1961vr} is it proved that any gradient vector field $-\nabla_e h$ can approximated
sufficiently $C^1$-close by $-\nabla_{g}f$, such   the associated flow  $\psi_{(f,g)}$
is Morse-Smale on $N$,  cf. \cite{Weber:2006tx}.
If $h$ is Morse one can  find a perturbation $g$ of the metric $e$ such that the flow $\psi_{(h,g)}$ is Morse-Smale, 
cf.~\cite{Abbondandolo:2005um}.
}
\end{remark}

From Lemma \ref{lem:inv2} it follows that for any Lyapunov function $\f$ for $(S,\phi)$ any flow $\psi_{(\f,g)}$,
$N$ is an isolating block with $\Inv(N,\psi_{(\f,g)}) \subset S$.
The choice of metric does not play a role and provides freedom in choosing isolated homotopies:
$$
\IMS(h;N) = \bigcup_{e} \IMS(h,e;N),
$$
which represents the set of Morse-Smale pairs $(f,g)$ which is isolated homotopic to $(h,e)$, for some Riemannian metric $e$ on $M$.

\begin{corollary}
\label{cor:MS-pairs}
Given a Lyapunov  function $\f\in \lyap(S,\psi)$, then there exists a Morse-Smale pair $(f, g) \in \IMS(\f;N)$ on $N$
with $f$ arbitrary $C^2$-close to $\f$.
\end{corollary}

\begin{proof}
If $\f\in \lyap(S,\phi)$ is a Lyapunov function for $(S,\phi)$, then for the gradient flow $\psi_{(\f,g)}$ it holds that $S_{(\f,g)} = \Inv(N,\psi) = \crit(\f)\cap N \subset S \subset \Int(N)$, with $g$ arbitrary, and thus by Proposition \ref{prop:morse}
there exist a Morse-Smale pair $(f, g)$, close to $(\f,g)$, such that $S_{(f,g)} = \Inv(N,\psi_{(f,g)}) \subset \Int(N)$
and $(f,g)$ is isolated homotopic to $(\f,g)$.

\end{proof}

\section{Morse homology}
\label{sec:MH}

From Lemma \ref{lem:inv2} it follows that $N$ is a isolating neighborhood for the gradient flow $\psi_{(\f,g)}$, for any Lyapunov function $\f\in \lyap(S,\phi)$, regardless of the chosen metric $g$. Corollary \ref{cor:MS-pairs} implies that there exists sufficiently $C^2$-close smooth function $f$, such that $(f,g)$  is a Morse-Smale pair on $N$ and $(f,g)$ is isolated homotopic to $(\f,g)$, i.e. $(f,g) \in \IMS(\f;N)$.

\subsection{Morse homology}
\label{subset:MH}
We follow the treatment of Morse homology given in \cite{Weber:2006tx}, applied to a Morse-Smale pair $(f,g)
\in \IMS(\f;N)$.
Define the moduli space by 
\begin{equation}
\label{eqn:moduli}
W_N(x,y)=W^u_{\rm loc}(x;N)\cap W^s_{\rm loc}(y;N).
\end{equation}
By Proposition \ref{prop:invariantset} and Corollary \ref{cor:MS-pairs} 
$$
S_{(f,g)} =\Inv(N, \psi_{(f,g)})=\bigcup_{x,y\in \crit(f)\cap N}W_N(x,y)\subset \Int(N),
$$
and thus $W_N(x,y) \subset \Int(N)$ for all $x,y \in \crit(f)\cap N$.
 The quotient $M_N(x,y) = W_N(x,y)/\mR$,  which can be identified with $W_N(x,y)\cap  f^{-1}(a)$, with $a\in\bigl(f(y),f(x)\bigr)$.

The `full' moduli spaces are given by $W(x,y) = W^u(x) \cap W^s(x)$. A connected component of $W(x,y)$ which intersects $W_N(x,y)$ is completely contained in $W_N(x,y)$. Since $\psi_{(f,g)}$ is Morse-Smale on $N$ 
the sets $M_N(x,y)$ are smooth submanifolds (without boundary) of dimension $\dim M_N(x,y) = \ind_{ f}(x)-\ind_{ f}(y)-1$, where $\ind_f$ is the Morse index of $x$, cf. \cite{Weber:2006tx}. The submanifolds $M_N(x,y)$ can be compactified by adjoining broken flow lines.  

If $\ind_{f}(x) = \ind_{f}(y) + 1$, then $M_N(x,y)$ is a zero-dimensional manifold without boundary, and since $\psi_{(f,g)}$ is Morse-Smale and $N$ is compact, the set $M_N(x,y)$ is a finite point set, cf. \cite{Weber:2006tx}. If $\ind_{f}(x) = \ind_{ f}(y) + 2$, then $M_N(x,y)$ is a disjoint union of copies of $\mS^{1}$ or $(0,1)$. 

The non-compactness in the components diffeomorphic to $(0,1)$ can be described as follows, see also Figure\ \ref{fig:brokenorbit}. Identify $M_N(x,y)$ with $W_N(x,y) \cap f^{-1}(a)$ (as before) and let $u_k \in M_N(x,y)$ be a sequence without any convergent subsequences. Then, there exist a critical $z\in \crit(f) \cap N$, with $f(y) <f(z)<f(x)$, reals $a^1,a^2$, with $f(z)<a^1<f(x)$ and
$f(y) <a^2<f(z)$, and times $t^1_k$ and $t^2_k$, such that
$f(\psi_{(f,g)}(t^1_k,u_k)) = a^1$ and $f(\psi_{(f,g)}(t^2_k,u_k)) = a^2$, such that
$$
\psi_{(f,g)}(t^1_k,u_k)\to v \in W_N(x,z)\cap f^{-1}(a^1), \quad\hbox{and}
$$  
$$
\psi_{(f,g)}(t^2_k,u_k)\to w \in W_N(z,y)\cap f^{-1}(a^2).
$$ 
We say that the non-compactness is
 represented by concatenations of two trajectories $x\to z$ and $z \to y$ with $\ind_f(z)=\ind_{ f}(x)-1$. 
 We write that $u_k \to (v,w)$ --- geometric convergence ---, and $(v,w)$ is a broken trajectory.
The following proposition adjusts Theorem 3.9 in \cite{Weber:2006tx}.

\begin{figure}
\def\svgwidth{.7\textwidth}
\hspace{3cm}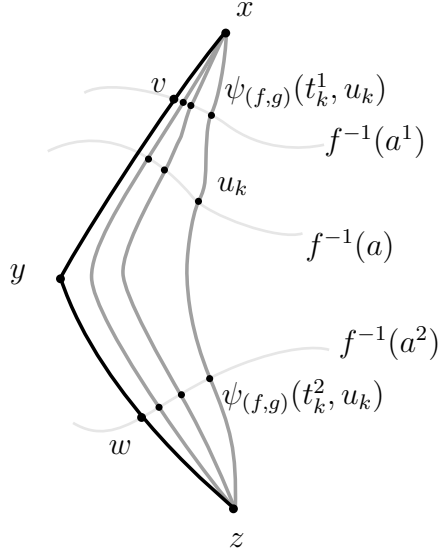
\caption{The sequence $u_k$ converges to the broken orbit $(v,w)$ as it approaches the end of the moduli space.}
\label{fig:brokenorbit}
\end{figure}
 \begin{lemma}[Restricted transitivity]
\label{prop:trans}
Suppose $x,y,z\in \crit(f) \cap N$ with $ind_{f}(x)=ind_{f}(z)+1$ and $\ind_f(z)=\ind_{ f}(y)+1$. 
Then there exist the (restricted) gluing embedding
$$
\#_\rho : \mR^+ \times M_N(x,z) \times M_N(z,y) \to M_N(x,y), \quad (\rho,v,w)\mapsto v\#_\rho w,
$$
such that $v\#_\rho w \to (v,w)$ in the geometric sense. Moreover, no sequences in $M_N(x,y)\setminus v\#_{\mR^+} w$
geometrically converge to $(v,w)$.
\end{lemma}

\begin{proof}
The flow $\psi_{(f,g)}$ is well-defined on $M$. In Morse homology we have  
 the restricted gluing embedding
 $$
 \#_\rho : \mR^+ \times M_N(x,z) \times M_N(z,y) \to M(x,y),
 $$
which maps a triple $(\rho,v,w)$ to an element $v\#_\rho w \in M(x,y)$, cf. \cite{Banyaga:2009ta}, \cite{Schwarz:1993wg}, \cite{Weber:2006tx}. The geometric convergence $v\#_\rho w \to (v,w)$,   as $\rho \to \infty$, implies that $\gamma_\rho = \psi_{(f,g)}(\mR, v\#_\rho w)$
gets arbitrary close to the images of $\psi_{(f,g)}(\mR, v)$ and $\psi_{(f,g)}(\mR, w)$.  Since $\psi_{(f,g)}(\mR, v) \subset \Int(N)$ and $\psi_{(f,g)}(\mR, w)\subset \Int(N)$
we obtain orbits $\gamma_\rho \subset \Int(N)$ for $\rho$ sufficiently large. This yields the embedding into $M_N(x,z)$.
\end{proof}

 Following the construction of Morse homology in \cite{Weber:2006tx} (see also \cite{Floer:1989tk}, \cite{Schwarz:1993wg}) we choose orientations of the unstable manifolds $W^u_{\rm loc}(x;N)$, for all $x\in \crit(f)\cap N$, and denote this choice by $\co$. Define chain groups and boundary operators
$$
C_{k}(f,N) = \bigoplus_{x\in \crit(f)\cap N\atop \ind_{f}(x)=k} \mZ \langle x\rangle,
$$
$$
\partial_{k}(f, g,N,\co): C_{k}(f,N) \to C_{k-1}(f,N),
$$
where
$$
\partial_{k}(f, g,N,\co) \langle x\rangle = \sum_{y\in \crit(f)\cap N\atop \ind_{f}(y)=k-1} n_N(x,y) \langle y\rangle,
$$
and $n_N(x,y)$ is the oriented intersection number of $W^{u}_{\rm loc}(x;N) \cap f^{-1}(a)$ and $W^{s}_{\rm loc}(y;N) \cap f^{-1}(a)$ with respect to $\co$.\footnote{Over $\mZ_{2}$ the intersection number is given by the number of elements in $ M_N(x,y)$ modulo 2.}

The isolation of $S_{(f,g)}$ in $N$ has strong implications for the boundary operator $\partial(f,g,N,\co)$.

\begin{lemma}
\label{thm:dsquared}
$\partial^2(f,g,N,\co)=0$.
\end{lemma}

\begin{proof}
Restricted transitivity in Proposition \ref{prop:trans} implies that all ends of the components
$M_N(x,y)$ that are diffeomorphic to $(0,1)$, are different and lie in $N$. The gluing map also implies that all broken trajectories in $N$ occur as unique ends of components of $M_N(x,y)$.
 By the same counting procedure as in \cite{Weber:2006tx} we obtain $\partial^2( f,g,N,\co)=0$.
\end{proof}

\begin{definition}
\label{defn:MCF}
{\em
A quadruple $\cQ = (f, g,N,\co)$ is called  a \emph{Morse-Conley-Floer quadruple} for $(S,\phi)$ if
\begin{enumerate}
\item [(i)] $N$ is an isolating neighborhood for $S$, i.e. $S = \Inv(N,\phi)$;
\item [(ii)]  $(f,g)\in \IMS(\f;N)$, for a Lyapunov function $\f\in \lyap(S,\phi)$;
\item [(iii)] $\co$ is choice of orientations of the unstable manifolds of the critical points of $f$ in $N$.
\end{enumerate}
}
\end{definition}

By Proposition \ref{thm:dsquared}  $\bigl( C_*(f,N), \partial_*(f,g,N,\co)\bigr)$ is a chain complex and the
the Morse homology of the quadruple $\cQ=(f, g,N,\co)$ is now defined as
$$
\HM_{*}(\cQ) := H_{*}(C_{*}(f,N),\partial_{*}(f,g,N,\co)).
$$

\subsection{Independence of Lyapunov functions and perturbations}
\label{subsec:indep}
The next step is to show that $\HM_{*}(\cQ)$ is independent of the choice of Lyapunov function,
metric and orientation, as well as the choice of perturbation $f$. In Section\ \ref{subsec:inverse} we show that the homology is also independent of the isolating neighborhood $N$.

\begin{theorem}
\label{prop:cont1}\label{thm:cont1}
Let $\cQ^\alpha=(f^\alpha,g^\alpha,N^\alpha,\co^\alpha)$ and $\cQ^\beta=(f^\beta,g^\beta,N^\beta,\co^\beta)$ be Morse-Conley-Floer quadruples for $(S,\phi)$, with the same isolating neighborhood $N=N^\alpha=N^\beta$. Then, there exist canonical isomorphisms
$$
\Phi^{\beta\alpha}_*:\HM_{*}(\cQ^\alpha) \xrightarrow{\cong} \HM_{*}(\cQ^\beta).
$$

\end{theorem}
The proof essentially follows the continuation theorem for Floer homology as given in \cite{Floer:1989tk}.
We sketch a proof based on the same arguments in \cite{Floer:1989wv} and \cite{Schwarz:1993wg} and the dynamical systems approach in \cite{Weber:2006tx}. The idea of the proof of this proposition is to construct higher dimensional systems which contain the complexes generated by both quadruples. The fundamental relation $\partial^2=0$ on the higher dimensional systems induces a map between the Morse homologies, which is then used to construct an isomorphism between the homologies of both quadruples.

Fix a metric $g$ on $M$. By assumption there are Lyapunov functions $\fal,\fbet \in \lyap(S,\phi)$ defined on the same isolating block $N$ and Morse-Smale pairs $(f^\alpha,g^\alpha) \in \IMS(\fal;N)$ and $(f^\beta,g^\beta)\in \IMS(\fbet;N)$. This implies that there exists a homotopy $(f^\alpha_\lambda,g^\alpha_\lambda)$ between $(f^\alpha,g^\alpha)$ and $(f_\phi^\alpha,g)$, and similarly there exists a homotopy  $(f^\beta_\lambda,g^\beta_\lambda)$ between $(f^\beta,g^\beta)$ and $(f_\phi^\beta,g)$. By Proposition \ref{prop:convex} the functions $$f_{\phi,\, \lambda}= (1-\lambda) \fal + \lambda \fbet,$$  are Lyapunov functions for $(S,\phi)$. Since $f_{\phi,\lambda} \in \lyap(S,\phi)$, for all $\lambda \in [0,1]$, it follows by Lemma \ref{lem:inv2} that $\Inv(N,\psi_{(f_{\phi,\lambda},g)}) =  \crit(f_{\phi,\lambda})\cap N \subset S$, where $\psi_{(f_{\phi,\lambda},g)}$ is the gradient flow 
of $(f_{\phi,\lambda},g)$, and thus $N$ is an isolating neighborhood for all $\psi_{(f_{\phi,\lambda},g)}$. We concatenate the homotopies and define 
$$
f_\lambda = \begin{cases}
f^\alpha_{3\lambda}      & \text{ for}~\lambda \in [0,\frac{1}{3}], \\
f_{\phi,3\lambda -1}      & \text{ for}~\lambda \in [\frac{1}{3},\frac{2}{3}],\\
f^\beta_{3-3\lambda}    & \text{ for}~\lambda \in [\frac{2}{3},1],
\end{cases}
$$
which is a piecewise smooth homotopy between $f^\alpha$ and $f^\beta$.
Similarly define
$$
g_\lambda = \begin{cases}
g^\alpha_{3\lambda}      & \text{ for}~\lambda \in [0,\frac{1}{3}], \\
g      & \text{ for}~\lambda \in [\frac{1}{3},\frac{2}{3}],\\
g^\beta_{3-3\lambda}    & \text{ for}~\lambda \in [\frac{2}{3},1],
\end{cases}
$$
which is a piecewise smooth homotopy of metric between $g^\alpha$ and $g^\beta$.
Both homotopies can be reparameterized to smooth homotopies via a diffeomorphism $\lambda \mapsto \alpha(\lambda)$,
with $\alpha'(\lambda)\ge 0$ and $\alpha^{(k)}(\lambda) =0$ for $\lambda=\frac{1}{3},\frac{2}{3}$ and for all  $k\ge 1$.
We denote the reparameterized homotopy again by $(f_\lambda,g_\lambda)$. The homotopies fit into the diagram
$$
\xymatrix{(f^\alpha,g^\alpha)\ar[r]^{(f_\lambda,g_\lambda)}\ar[d]_{(f_\lambda^\alpha,g_\lambda^\alpha)}&(f^\beta,g^\beta)\ar[d]^{(f_\lambda^\beta,g_\lambda^\beta)}\\
(f_\phi^\alpha,g)\ar[r]_{(f_{\phi,\lambda},g)}&(f_\phi^\beta,g)}
$$
The  flow of $(f_\lambda,g_\lambda)$ is denoted by $\psi_{(f_\lambda,g_\lambda)}$. By assumption and construction 
$\Inv(N,\psi_{(f_\lambda,g_\lambda)}) \subset \Int(N)$ for all $\lambda \in [0,1]$ and therefore $(f^\alpha,g^\alpha)
\in \IMS(f^\beta,g^\beta;N)$ and vice versa.
Note that $ f_{\lambda}$ is not necessarily Morse for each $\lambda \in (0,1)$.
At $\lambda=0,1$  the flows $\psi_{(f_\lambda,g_\lambda)}$ are Morse-Smale flows on $N$ by assumption.

Let $r>0$ and $0<\delta < \frac{1}{4}$, and define the function $F:M\times \mS^1\rightarrow \mR$
\begin{equation}
\label{eq:continuation}
F(x,\mu) =  f_{\omega(\mu)}(x) + r \bigl[ 1+\cos(\pi \mu)\bigr],
\end{equation}
where $\omega: \mR \to [0,1]$ is a smooth, even,  2-periodic function with the following properties
on the interval $[-1,1]$:
 $\omega(\mu)=0$, for $-\delta<\mu<\delta$, $\omega(\mu)=1$, for $-1 < \delta <-1+\delta$ and $1-\delta < \mu <1$, and
$\omega'(\mu) <0$ on $(-1+\delta,-\delta)$ and $\omega'(\mu)>0$ on $(\delta,1-\delta)$. By the identifying $\mS^{1}$ with $\mR/2\mZ$, the function $\omega$ descents to a function on $\mS^1$, which is denoted by the same symbol. We consider the product metric $G^\times$ on $M\times \mS^{1}$ defined by
$$
G_{(x,\mu)}^\times  = (g_{\omega(\mu)})_x \oplus \frac{1}{\kappa} d\mu^2,\quad \kappa >0.
$$
For the proof of the theorem it would suffice to take $\kappa=1$. However, in the proof of Proposition\ \ref{prop:continuation} we do need the parameter $\kappa$. In order to use the current proof ad verbatim, we introduce the parameter.
\begin{lemma}
If
$$
r>\frac{\max_{\mu\in\mS^1,x\in N}|\omega'(\mu) \left[\partial_\lambda f_\lambda(x)\bigr|_{\lambda=\omega(\mu)}\right]|}{\pi\sin(\pi\delta)},
$$
then for each critical point $(x,\mu)\in\crit(F)\cap (N\times\mS^1)$, either $\mu=0$ and $\ind_{F}(x,0)=\ind_{f^\alpha}(x)+1$, or $\mu=1$ and $\ind_{F}(x,1)=\ind_{ f^\beta}(x)$.
In particular, $F$ is a Morse function on $N\times \mS^1$ and
\begin{equation}
\label{eq:split}
C_k(F,N\times \mS^1)\cong C_{k-1}( f^\alpha,N)\oplus C_k( f^\beta,N).
\end{equation}
\end{lemma}
\begin{proof}
The $G$-gradient vector field of $F$ on $M\times \mS^{1}$ is given by
\begin{align*}
\nabla_{G^\times} F &(x,\mu) =\\
&\nabla_{g_{\omega(\mu)}}  f_{\omega(\mu)}(x) +\kappa\Bigl( \omega'(\mu) \bigl[ \partial_\lambda f_\lambda(x)\bigl|_{\lambda=\omega(\mu)} \bigr]- r \pi \sin(\pi\mu)\Bigr){\partial_\mu}.
\end{align*}
by the choice on $r$ the second term is only zero on $B\times \mS^1$ if $\mu=0,1$. Then the critical points at $\mu=0$ correspond to critical points of $f^\alpha$, but there is one extra unstable direction, corresponding to ${\partial_\mu}$. At $\mu=1$ the critical points come from critical points of $f^\beta$, and there is one extra stable direction, which does not contribute to the Morse index.
\end{proof}

This outlines the construction of the higher dimensional system. From now on we assume $r$ specified as above. 
\begin{figure}
\def\svgwidth{.7\textwidth}
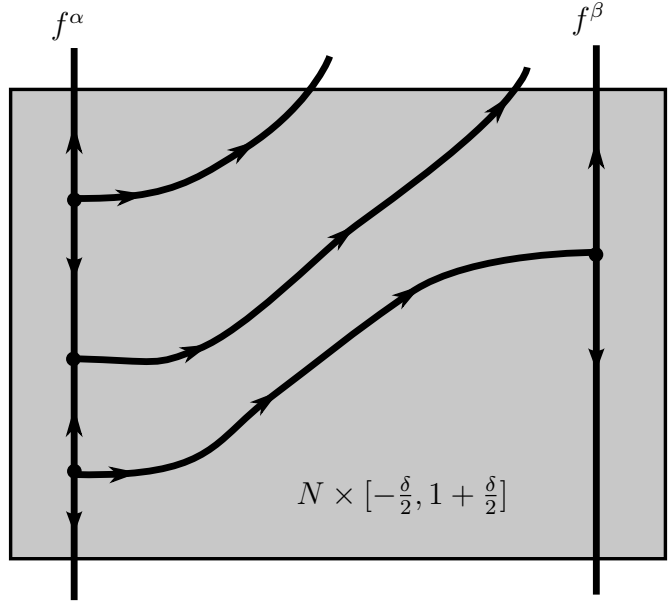
\caption{The continuation map $\Phi_*^{\beta\alpha}:\HM_*(\cQ^\alpha)\rightarrow\HM_*(\cQ^\beta)$ is induced by a higher dimensional gradient system with Morse function of Equation\ \bref{eq:continuation}. }
\label{fig:continuation}
\end{figure}

\begin{proof}[Proof of Theorem \ref{prop:cont1}]
Consider the negative gradient flow 
$\Psi^\times_\kappa: \mR \times M \times \mS^1 \to M\times \mS^{1}$ is generated by the vector field $-\nabla_{G^\times} F$, cf.\ Figure\ \ref{fig:continuation}. 
For $\kappa =0$ the flow $\Psi^\times_0$ is a product flow and $N\times \mS^1$ is an isolating neighborhood for 
$\{ (x,\lambda)~|~x\in S_{(f_\lambda,g_\lambda)}\}$. Since isolation is preserved under small perturbations,
$N\times \mS^1$ is also an isolating neighborhood for $\Psi^\times_\kappa$, for all $\kappa>0$ sufficiently small.
Note that $C=N\times[\frac{-\delta}{2},1+\frac{\delta}{2}]$ is also an isolating neighborhood of the gradient flow $\Psi^\times_\kappa$. The flow $\Psi^\times$ is not Morse-Smale on $C$, but since $F$ is Morse, a small perturbation $G$ of the product  metric $G^\times$, makes $(F,G)$ a Morse-Smale pair and therefore the associated negative $G$-gradient flow $\Psi$ is Morse-Smale, without destroying the isolation property of $C$. The connections at $\mu=0$ and $\mu=1$ are stable, because these are Morse-Smale for $\kappa=0$. 
This implies that we can still denote the boundary operators at $\mu=0,1$ by $\partial^\alpha=\partial(f^\alpha,g^\alpha,N,\co^\alpha)$
and $\partial^\beta=\partial(f^\beta,g^\beta,N,\co^\beta)$ respectively, since the perturbed metric at the end point do not change the connections and their oriented intersection numbers.
We now choose the following orientation 
$$\cO=\bigl(\bigl[\partial\mu\bigr]\oplus \co_0\bigr)\cup\co_1,$$ 
for the unstable manifolds of the critical points of $F$ in $C$.

By Theorem \ref{thm:dsquared} the Morse homology of the Morse-Conley-Floer  quadruple $(F,G,C,\cO)$ is well-defined. With regards to the splitting of Equation \bref{eq:split} the boundary operator $\Delta_k=\partial_k(F, G,C,\cO)$ is given by
$$
\Delta_k = \left(\begin{array}{cc}-\partial_{k-1}^\alpha & 0 \\ \Phi^{\beta\alpha}_{k-1} & \partial_k^\beta\end{array}\right).
$$
The map $\Phi^{\beta\alpha}$ counts gradient flow lines from $(x,0)$ to $(y,1)$ with $\ind_F(x,0)=\ind_F(y,1)+1$. Theorem~\ref{thm:dsquared} gives $\Delta_{k-1} \circ \Delta_{k} =0$, from which we derive that $\Phi_*^{\beta\alpha}$ is a chain map
$$
\Phi^{\beta\alpha}_{k-2}\circ \partial^\alpha_{k-1}=\partial^\beta_{k-1}\circ \Phi^{\beta\alpha}_{k-1},
$$
hence $\Phi^{\beta\alpha}_*$ descents to a map in homology $\Phi^{\beta\alpha}_*:\HM_*( \cQ^\alpha)\rightarrow \HM_*(\cQ^\beta)$, which we denote by the same symbol. 
The arguments in~\cite{Weber:2006tx} can be used now to show that the maps $\Psi_*^{\beta\alpha}$ only
depend on the `end points' $\cQ^\alpha$ and $\cQ^\beta$, and not on the particular homotopy.
It is obvious that if $\cQ^\alpha=\cQ^\beta$ this map is the identity, both on the chain and homology level. By looking at a higher dimensional system, we establish that, if $\cQ^\gamma=( f^\gamma,g^\gamma,N^\gamma,\co^\gamma)$, with $N^\gamma=N$ is a third Morse-Conley-Floer quadruple, then the induced maps satisfy the functorial relation
$$
\Phi_*^{\gamma\alpha}=\Phi_*^{\gamma\beta}\circ\Phi_*^{\beta\alpha},
$$
on homology. The argument is identical to \cite{Weber:2006tx}, which we therefore omit. Taking $( f^\gamma,g^\gamma,N,\co^\gamma)=( f^\alpha,g^\alpha,N,\co^\alpha)$, along with the observation that $\Phi^{\alpha\alpha}=\id$, shows that the maps are isomorphisms.
\end{proof}

Let $\cQ^\alpha,\cQ^\beta,\cQ^\gamma$ be Morse-Conley-Floer quadruples for $(S,\phi)$ with $N=N^\alpha=N^\beta=N^\gamma$. This gives the above defined isomorphisms:
\begin{align*}
\Phi_*^{\beta\alpha}: \HM_*(\cQ^\alpha) &\to \HM_*( \cQ^\beta),\\
\Phi_*^{\gamma\beta}: \HM_*(\cQ^\beta) &\to \HM_*( \cQ^\gamma),\\
\Phi_*^{\gamma\alpha}: \HM_*(\cQ^\alpha) &\to \HM_*(\cQ^\gamma).
\end{align*}
The proof of the previous proposition shows that the following functorial relation holds. 
\begin{theorem}
\label{prop:cont3}
It holds that
$\Phi_*^{\gamma\alpha} = \Phi_*^{\gamma\beta}\circ\Phi_*^{\beta\alpha}$.
\end{theorem}

\subsection{Independence of isolating blocks}
\label{subsec:inverse}
In section we show that Morse homology of a Morse-Conley-Floer quadruple is independent of the choice of isolating block $N$ for the isolated invariant set $S$.
\begin{theorem}
\label{prop:cont2}\label{thm:cont2}
Let $\cQ^\alpha = (f^\alpha,g^\alpha,N^\alpha,\co^\alpha)$ and $\cQ^\beta = (f^\beta,g^\beta,N^\beta,\co^\beta)$ be Morse-Conley-Floer quadruples for $(S,\phi)$. Then, there are canonical isomorphisms
$$
\Phi^{\beta\alpha}_*: \HM_*(\cQ^\alpha) \xrightarrow{\cong} \HM_*(\cQ^\beta).
$$
\end{theorem}

\begin{proof}
By assumption there exist Lyapunov functions $\fal,\fbet$   on $N^\alpha,N^\beta$
and Morse-Smale pairs $(f^\alpha,g^\alpha) \in \IMS(\fal;N^\alpha)$ and $(f^\beta,g^\beta)\in \IMS(\fbet;N^\beta)$ be Morse-Smale pairs. We cannot directly compare the complexes as in the previous proposition, however $N=N^\alpha\cap N^\beta$ is an isolating neighborhood. The restrictions of $\fal,\fbet$ to $N$ are Lyapunov functions. 
Since $(f^\alpha,g^\alpha) \in \IMS(\fal;N^\alpha)$, there exists a smooth function $f^\alpha_\epsilon: M \to \mR$, and Riemannian metric $g^\alpha_\epsilon$,
with $\Inv(\psi_{(f^\alpha_\epsilon,g^\alpha_\epsilon)},N^\alpha) \subset \Int(N)$,
such that 
$$
(f^\alpha,g^\alpha) \in  \IMS(f^\alpha_\epsilon;N^\alpha), \quad {\rm and}\quad (f^\alpha_\epsilon,g^\alpha_\epsilon) \in \IMS(\fal;N).
$$ 
Similarly, there exists a smooth function 
$f^\beta_\epsilon: M \to \mR$, and a Riemannian metric $g^\beta_\epsilon$, with  $\Inv(\psi_{(f^\beta_\epsilon,g^\beta_\epsilon)},N^\beta) \subset \Int(N)$,
such that 
$$
(f^\beta,g^\beta) \in \IMS(f^\beta_\epsilon;N^\beta)\quad {\rm and}\quad  (f^\beta_\epsilon,g^\beta_\epsilon) \in \IMS(\fbet;N).
$$ 
This yields the following diagram of isomorphisms
$$
\xymatrix@C=.25cm{\HM_*(\cQ^\alpha)\ar[r]^-{\Psi_*^\alpha}\ar[d]^{\Phi_*^{\beta\alpha}}&\HM_*(f^\alpha_\epsilon,g^\alpha_\epsilon,N^\alpha,\co^\alpha)\ar[r]^-{=}&\HM_*( f^\alpha_\epsilon,g^\alpha_\epsilon,N,\co^\alpha)\ar[d]^{\Phi_*^{\beta\alpha,\epsilon}}\\
\HM_*(\cQ^\beta)\ar[r]^-{\Psi_*^\beta}&\HM_*( f^\beta_\epsilon,g^\beta_\epsilon,N^\beta,\co^\beta)\ar[r]^-{=}&\HM_*(f^\beta_\epsilon,g^\beta_\epsilon,N,\co^\beta),
}
$$
where $\Psi_*^\alpha$, $\Psi_*^\beta$ and $\Phi_*^{\beta\alpha,\epsilon}$ are given by Theorem \ref{thm:cont1} and $\Phi_*^{\beta\alpha} = (\Psi_*^\beta)^{-1} \circ \Phi_*^{\beta\alpha,\epsilon}\circ \Psi_*^\alpha$.
The equalities in the above diagram follow from the isolation properties of $f^\alpha_\epsilon$ and $f^\beta_\epsilon$ with respect to $N$!

Suppose we use different functions ${f_\epsilon^\alpha}', {f_\epsilon^\beta}'$ and metric ${g^\alpha_\epsilon}',{g^\beta_\epsilon}'$
as above. Then, we obtain isomorphisms ${\Psi_*^\alpha}'$, ${\Psi_*^\beta}'$ and $\Phi_*^{\beta\alpha,\epsilon'}$, which yields the isomorphisms $\Phi_*^{\beta\alpha'} = ({\Psi_*^\beta}')^{-1} \circ \Phi_*^{\beta\alpha,\epsilon'}\circ {\Psi_*^\alpha}'$
between $\HM_*(\cQ^\alpha)$ and $\HM_*(\cQ^\beta)$.
Define
$$
\Phi^{\alpha'\alpha}_* = {\Psi_*^\alpha}'\circ (\Psi_*^\alpha)^{-1} : \HM_*(f^\alpha_\epsilon,g^\alpha_\epsilon,N,\co^\alpha)
\to \HM_*({f^\alpha_\epsilon}',{g^\alpha_\epsilon}',N,\co^\alpha),
$$
and similarly
$$
\Phi^{\beta'\beta}_* = {\Psi_*^\beta}'\circ (\Psi_*^\beta)^{-1} : \HM_*(f^\beta_\epsilon,g^\beta_\epsilon,N,\co^\beta)
\to \HM_*({f^\beta_\epsilon}',{g^\beta_\epsilon}',N,\co^\beta),
$$
where we use the equalities $\HM_*(f^\alpha_\epsilon,g^\alpha_\epsilon,N^\alpha,\co^\alpha) = \HM_*(f^\alpha_\epsilon,g^\alpha_\epsilon,N,\co^\alpha)$ and 
$\HM_*(f^\beta_\epsilon,g^\beta_\epsilon,N^\beta,\co^\beta) = \HM_*(f^\beta_\epsilon,g^\beta_\epsilon,N,\co^\beta)$.
From the Morse homologies of Morse-Conley-Floer quadruples with $N$ fixed and Theorem \ref{thm:cont1}
we obtain the  following commutative diagram:
$$
\xymatrix@C=.25cm{\HM_*(f^\alpha_\epsilon,g^\alpha_\epsilon,N,\co^\alpha)\ar[d]^{\Phi^{\alpha'\alpha}_*}\ar[r]^-{\Phi_*^{\beta\alpha,\epsilon}}&\HM_*( f^\beta_\epsilon,g^\beta_\epsilon,N,\co^\beta)\ar[d]^{\Phi^{\beta'\beta}_*}\\
\HM_*({f^\alpha_\epsilon}',{g^\alpha_\epsilon}',N,\co^\alpha)\ar[r]^-{\Phi^{\beta\alpha,\epsilon'}_*}&\HM_*({f^\beta_\epsilon}',{g^\beta_\epsilon}',N,\co^\beta).
}
$$
Combining the isomorphisms yields
$$
\Phi_*^{\beta\alpha'} = (\Psi_*^{\beta})^{-1} \circ (\Phi_*^{\beta'\beta})^{-1}\circ
\Psi_*^{\beta'\beta} \circ \Phi_*^{\beta\alpha,\epsilon}\circ (\Psi_*^{\alpha'\alpha})^{-1}
\circ \Psi_*^{\alpha'\alpha}\circ \Psi_*^\alpha = \Phi_*^{\beta\alpha},
$$
which proves the independence of isolating blocks.
\end{proof}

The isomorphisms in Theorem \ref{prop:cont2} satisfy the composition law.
If we consider Morse-Smale quadruples $\cQ^\alpha,\cQ^\beta,\cQ^\gamma$, we find isomorphisms:
 \begin{align*}
\Phi_*^{\beta\alpha}: \HM_*(\cQ^\alpha) &\to \HM_*(\cQ^\beta),\\
\Phi_*^{\gamma\beta}: \HM_*(\cQ^\beta) &\to \HM_*( \cQ^\gamma),\\
\Phi_*^{\gamma\alpha}: \HM_*( \cQ^\alpha) &\to \HM_*(\cQ^\gamma).
\end{align*}
\begin{theorem}
\label{prop:cont4}
The composition law holds:
$\Phi^{\gamma\alpha}_* = \Phi^{\gamma\beta}_*\circ\Phi^{\beta\alpha}_*$.
\end{theorem}
\begin{proof}
The intersection $N=N^\alpha\cap N^\beta\cap N^\gamma$ is an isolating neighborhood. Proceeding as in the proof of the previous theorem, and using Theorem~\ref{prop:cont3} on the isolating neighborhood $N$, shows that the composition law holds. 
\end{proof}
\section{Morse-Conley-Floer homology}
\label{sec:CF-index}

Theorem \ref{prop:cont4} implies that the Morse homology $\HM_*(\cQ^\alpha)$ 
is an inverse system with respect to the canonical isomorphisms 
$$
\Phi_*^{\beta\alpha}: \HM_*(\cQ^\alpha) \xrightarrow{\cong} \HM_*(\cQ^\beta).
$$
This leads to the following definition. 
  
\begin{definition}
\label{defn:CF-index}
{\em
Let $S$ be an isolated invariant set for $\phi$. The \emph{Morse-Conley-Floer homology} of $(S,\phi)$ is defined by 
\begin{equation}
\label{eqn:CF-index2}
\HI_*(S,\phi) := \varprojlim \HM_*(\cQ)
\end{equation}
with respect to Morse-Conley-Floer quadruples $\cQ$ for $(S,\phi)$ and the associated canonical  isomorphisms.

}
\end{definition}

\begin{proposition}
Suppose $N$ is an isolating neighborhood with $S=\Inv(N,\phi)=\emptyset$, then $HI_*(S,\phi)=0$. 
\end{proposition}
\begin{proof}
Let $\f$ be a Lyapunov function for $S$ on $N$. Since $\crit \f \cap N\subset S$ by Lemma~\ref{lem:crit1}, we have that $\f$ has no critical points in $N$ and $\f$ is automatically Morse-Smale on $N$. It follows that $HM_*(\f,g,\co,N)=0$, for any metric $g$, and therefore it follows that $HI_*(S,\phi)=0$. 
\end{proof}

Local continuation gives invariance of the Morse-Conley-Floer homology with respect to nearby flows. Consider a family of flows  $\phi_\lambda$, with $\lambda \in[0,1]$, and the denote the flow at $\lambda =0$ by $\phi^\alpha$ and
the flow at $\lambda =1$ by $\phi^\beta$.

\begin{proposition}
\label{prop:continuation}
Let $N\subset M$ be compact, and let $\phi_\lambda$ a smooth family of flows 
as described above
such that $N$ is an isolating neighborhood for $\phi_\lambda$ for all $\lambda \in[0,1]$, i.e.
$S_\lambda = \Inv(N,\phi_\lambda) \subset \Int(N)$.
 Set $S^\alpha=\Inv(N,\phi^\alpha)$ and $S^\beta=\Inv(N,\phi^\beta)$. Then there exists isomorphisms
$$
\HI_*(S^\alpha,\phi^\alpha)\cong \HI_*(S^\beta,\phi^\beta).
$$ 
\end{proposition}
\begin{proof}
The first step in the proof is to construct a homotopy of Lyapunov functions.
As before we take $\mu\in \mS^1 \cong \mR/2\mZ$ and set $\lambda = \omega(\mu)$, where $\omega$ is defined in Section \ref{subsec:indep}.
Define the product flow $\Phi:\mR\times M\times \mS^1\rightarrow M\times\mS^1$ by $\Phi(t,x,\mu)=(\phi_{\omega(\mu)}(t,x),\mu)$. By assumption
$N \times \mS^1$ is an isolating neighborhood for $\Phi$, and hence contains an isolating block $B\subset N \times \mS^1$ with
$\Inv(B,\Phi) = \bigcup_{\mu\in [0,1]} S_{\omega(\mu)}$.
By Proposition~\ref{prop:lyap1} there exists a Lyapunov function $F_\Phi: M\times\mS^1\to \mR$ for the flow $\Phi$ with Lyapunov property respect to $B$. The fibers $B_\mu=\{x\in M~|~(x,\mu)\in B\}$ are   isolating neighborhoods for  the flows $\phi_{\omega(\mu)}$ and the functions $f^\mu_{\phi_{\omega(\mu)}} (x)= F_\Phi(x,\mu)$ are Lyapunov functions for $\phi_{\omega(\mu)}$ on $B_\mu$. Denote the Lyapunov functions at $\mu=0,1$ by $f^\alpha_{\phi^\alpha}$ and $f^\beta_{\phi^\beta}$  respectively and set  $B^\alpha = B_0$ and $B^\beta= B_1$. We have now established a homotopy of Lyapunov functions $f^\mu_{\phi_{\omega(\mu)}}$.

Choose a metric $g$ on $M$. Since by Corollary \ref{cor:MS-pairs}, $\IMS(f^\alpha_{\phi^\alpha};B^\alpha)\not = \varnothing$ and $\IMS(f^\beta_{\phi^\beta};B^\beta)\not = \varnothing$, then there  exist Morse-Smale pairs $(f^\alpha,g^\alpha)$ and $(f^\beta,g^\beta)$ on $B^\alpha$ and $B^\beta$ respectively.
The associated homotopies are $(f^\alpha_\lambda,g^\alpha_\lambda)$ between
$(f^\alpha,g^\alpha)$ and $(f^\alpha_{\phi^\alpha},g)$, and similarly
$(f^\beta_\lambda,g^\beta_\lambda)$ between $(f^\beta,g^\beta)$ and $(f^\beta_{\phi^\beta},g)$. 
Define the homotopy
$$
f_\lambda = \begin{cases}
f^\alpha_{3\lambda}      & \text{ for}~\lambda \in [0,\frac{1}{3}], \\
f^{3\lambda -1}_{\phi_{\omega(3\lambda -1)}}      & \text{ for}~\lambda \in [\frac{1}{3},\frac{2}{3}],\\
f^\beta_{3-3\lambda}    & \text{ for}~\lambda \in [\frac{2}{3},1],
\end{cases}
$$
which is a piecewise smooth homotopy between $f^\alpha$ and $f^\beta$ and
similarly
$$
g_\lambda = \begin{cases}
g^\alpha_{3\lambda}      & \text{ for}~\lambda \in [0,\frac{1}{3}], \\
g      & \text{ for}~\lambda \in [\frac{1}{3},\frac{2}{3}],\\
g^\beta_{3-3\lambda}    & \text{ for}~\lambda \in [\frac{2}{3},1].
\end{cases}
$$
Furthermore we have the isolating neighborhood,
$$
N_\lambda = \begin{cases}
B^\alpha      & \text{ for}~\lambda \in [0,\frac{1}{3}], \\
B_{3\lambda -1}     & \text{ for}~\lambda \in [\frac{1}{3},\frac{2}{3}],\\
B^\beta    & \text{ for}~\lambda \in [\frac{2}{3},1].
\end{cases}
$$
As before we tacitly assume a reparametrization of the variable $\lambda$ to make the above homotopies smooth. We denote the negative gradient flows of the homotopy $(f_\lambda,g_\lambda)$ by $\psi^\lambda_{(f_\lambda,g_\lambda)}$ and
by assumption and construction (cf.~Lemma \ref{lem:inv2}) $\Inv(N_\lambda, \psi^\lambda_{(f_\lambda,g_\lambda)}) \subset \Int(N_\lambda)$.

Take $\mu\in \mS^1 \cong \mR/2\mZ$ and set $\lambda = \omega(\mu)$, where $\omega$ is defined in Section \ref{subsec:indep}
and consider the negative gradient flow $\Psi^\times_\kappa: \mR \times M \times \mS^1 \to M\times \mS^1$, given by the
function
$
F(x,\mu) =  f_{\omega(\mu)}(x) + r \bigl[ 1+\cos(\pi \mu)\bigr]
$ 
and product metric 
$G_{(x,\mu)}^\times  = (g_{\omega(\mu)})_x \oplus \frac{1}{\kappa} d\mu^2$.
For $r$ large and $\kappa$ small, and a small perturbation $G$ of the metric $G^\times$ we obtain a Morse-Smale pair $(F,G)$  on $C = \{(x,\mu)~:~ x\in N_{\omega(\mu)}\}$.
We can now repeat the proof of Theorem \ref{thm:cont1} to conclude isomorphisms
$$
\HM_*(\cQ^\alpha)\cong \HM_*(\cQ^\beta).
$$
If in the above construction we choose different Morse-Conley-Floer quadruples ${\cQ^\alpha}'$ and ${\cQ^\beta}'$ we obtain the commutative square of canonical isomorphisms
$$
\xymatrix{
\HM_*(\cQ^\alpha)\ar[d]^{\cong}\ar[r]^\cong&\HM_*({\cQ^\alpha}')\ar[d]^{\cong}\\
\HM_*(\cQ^\beta)\ar[r]^{\cong}&\HM_*({\cQ^\beta}'),
}
$$
which shows that  Morse-Conley-Floer homologies are isomorphic.
\end{proof}

The Morse-Conley-Floer homology is invariant under global deformations if the flows are related by continuation. This property allows for the computation of the index in highly non-linear situations, by continuing the flow to a flow that is better understood. 

\begin{definition}
\label{def:globalcont}
{\em
Isolated invariant sets $S^\alpha,S^\beta$ for the flows $\phi^\alpha$ and $\phi^\beta$ respectively, are said to \emph{related by continuation} if there exists a smooth homotopy $\phi_\lambda$ of flows, and a partition $\lambda_i$ of $[0,1]$, i.e.
$$
0=\lambda_0<\lambda_1<\ldots<\lambda_n=1,
$$
along with compact sets $N^\alpha=N_0 ,\ldots, N_{n-1} = N^\beta$, such that $N_i$ are  isolating neighborhoods for all flows $\phi_\lambda$, for all $\lambda_i\leq\lambda\leq\lambda_{i+1}$, and $\Inv(N_{i-1},\phi_{\lambda_i})=\Inv(N_i,\phi_{\lambda_i})$, $\Inv(N^\alpha,\phi^\alpha)=S^\alpha$, and $\Inv(N^\beta,\phi^\beta) =S^\beta$.
}
\end{definition}

Composing the isomorphisms yields a global continuation theorem.

\begin{theorem}
\label{thm:continuation}
Let $(S^\alpha,\phi^\alpha)$ and $(S^\beta,\phi^\beta)$ be related by continuation. Then there exists canonical isomorphisms
$$
\HI_*(S^\alpha,\phi^\alpha)\cong \HI_*(S^\beta,\phi^\beta).
$$
\end{theorem}

\section{Morse decompositions and connection matrices}
\label{sec:MD-CM}
If the dynamics restricted to an isolated invariant set $S$ is not recurrent a decomposition extracting gradient dynamics exists and
leads to the concept of \emph{Morse decomposition}, which generalizes the attractor-repeller pair.
\begin{definition}
\label{defn:MD}
{\em
Let $S = \Inv(N,\phi)$ be an isolated invariant set.
A family $\sfS = \{S_i\}_{i\in I}$, indexed by a finite poset $(I,\le)$, consisting of non-empty, compact, pairwise disjoint, invariant subsets $S_i\subset S$,  is 
a  \emph{Morse decomposition} for $S$ if, for every $x\in S\setminus (\bigcup_{i\in I} S_i)$, there exists $i<j$ such that
$$
\alpha(x) \subset S_j\quad {\rm  and} \quad \omega(x) \subset S_i.
$$
The sets $S_i \subset S$ are called  \emph{Morse sets}.  
 The set of all  Morse decompositions of  $S$ under $\phi$ is denoted by
$\sMD(S,\phi)$. 
}
\end{definition}
\begin{figure}
\def\svgwidth{.65\textwidth}
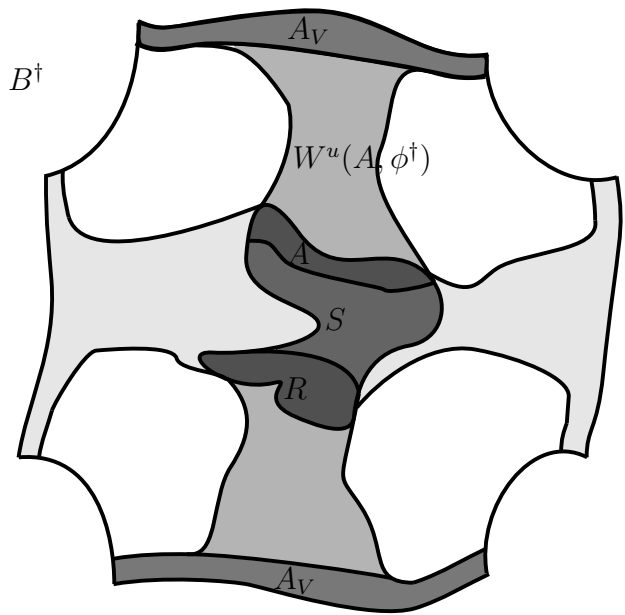
\caption{The Morse decomposition of $B^\dagger$, given in Lemma\ \ref{lem:att-in-att}. In Figure\ \ref{fig:bdaggerauv} the sets $A_U$ and $A^*_V$ are depicted.}
\label{fig:morsedec}
\end{figure}
Let $B\subset N$ be an isolating block for $S$. In the proof of Proposition \ref{prop:lyap1} we 
constructed a flow $\phi^\dagger: \mR\times B^\dagger \to B^\dagger$ as an extension of $\phi$, where $B^\dagger$
is a compact submanifold with piecewise smooth boundary.
Recall the attractor $A_V$ and the repeller $A^*_U$. From the duals we have $S = A_U \cap A_V^*$, cf. Figure\ \ref{fig:bdaggerauv}. Now consider an attractor $A$ in $S$, and its dual repeller $R=A|_S^*$ in $S$. Because $A,R$ are isolated invariant sets in $S$, and $S$ is isolated invariant in $B^\dagger$, $A,R$ are also isolated invariant on $B^\dagger$ for the flow $\phi^\dagger$, see\ \cite{Conley-Zehnder-2}. By a similar reasoning $A$, and $R$ are also isolated invariant sets on $M$ for the flow $\phi$. The above situation is sketched in Figure\ \ref{fig:morsedec}.

\begin{lemma}[cf.~\cite{KMV,KMV2}]
\label{lem:att-in-att}
The isolated invariant set $B^\dagger$ has a natural Morse decomposition $A_V<S<A^*_U$.
The invariant set
$$
\aAz =  A_V\cup W^u(A,\phi^\dagger)   \subset \aA = A_V \cup W^u(S,\phi^\dagger),
$$
is an attractor in $\aA$ and therefore an attractor in $B^\dagger$.
\end{lemma}

\begin{proof}
From the proof of Proposition \ref{prop:lyap1} it follows that $A_V < S$ is an attractor-repeller in $\aA$.
The set $R$ is a repeller in $S$ and therefore is a repeller in $\aA$. This provides the  Morse
decomposition
$$
A_V < A < R,
$$
for $\aA$. From this we derive that $\aAz< R$, with
$$
\aAz = A_V \cup W^u(A,\phi^\dagger),
$$
is an attractor-repeller pair in $\aA$ and $\aAz$ is an attractor in $\aA$ and thus $\aAz$ is an
attractor in $B^\dagger$, which proves the lemma.
\end{proof}

For a given isolated invariant sets $S$, with isolating block $B$, Lemma \ref{lem:att-in-att} yields the following filtration 
of attractors in $B^\dagger$ of the extended flow $\phi^\dagger$:
$$
\varnothing \subset A_V \subset \aAz \subset \aA \subset B^\dagger.
$$
The associated Morse decomposition of $B^\dagger$ is given by
$$
A_V < A < R < A^*_U,
$$
indeed, 
$$
A_V^* \cap \aAz = A_V|^*_{\aAz} = A,\quad {\rm and}\quad \aAz^*\cap \aA = \aAz|^*_{\aA} = R.
$$
From the proof of Proposition \ref{prop:lyap1} we have the Lyapunov functions
$$
f_{A_V}, f_{\aAz},\quad {\rm and}\quad  f_{\aA}.
$$
Consider the positive linear combination
$$
f_{\phi^\dagger}^\epsilon = \lambda f_{A_V} +\mu  f_{\aA} +\epsilon  f_{\aAz} 
$$
which is constant on the Morse sets $A_V$, $A$, $R$ and $A^*_U$ and 
decreases along orbits of $\phi^\dagger$. The values at the Morse sets are:
$f_{\phi^\dagger}^\epsilon|_{A_V} = 0$, $f_{\phi^\dagger}^\epsilon|_A = \lambda$, $f_{\phi^\dagger}^\epsilon|_R = \lambda + \epsilon$ and $f_{\phi^\dagger}^\epsilon|_{A^*_U} = \lambda+\mu+\epsilon$. The function $f_{\phi^\dagger}^\epsilon$ is also Lyapunov function for $(A\cup R, \phi)$, because it satisfies the Lyapunov property on a suitable chosen isolating neighborhood of $A\cup R$ for the flow $\phi$, which we therefore denote by $f_\phi^\epsilon$. Moreover, $f_{\phi}^\epsilon$ is an $\epsilon$-perturbation of $h^0_\epsilon =  \lambda f_{A_V} +\mu  f_{\aA}\in \lyap(S,\phi)$.

\begin{lemma}
\label{lem:poly-2}
Let $S$ be an isolated invariant set and let $A\subset S$
be an attractor in $S$. Then,
\begin{equation}
\label{eqn:MRfor-S}
 P_t(A,\phi) + P_t(R,\phi)  = P_t(S,\phi) +  (1+t) Q_t,
\end{equation}
where $P_t(A,\phi), P_t(R,\phi)$ and $P_t(S,\phi)$ are the Poincar\'e polynomials
of the associated  Morse-Conley-Floer homologies, and $Q_t$ is a polynomial with non-negative coefficients.
\end{lemma}

\begin{proof} 
Let $g$ be a Riemannian metric on $M$ and let $f_\phi^\epsilon$ be the Lyapunov function for the attractor-repeller pair $(A,R)$ as described above.
By Proposition \ref{prop:morse} there exists $\epsilon$-$C^2$ close perturbation $f$ of $f_\phi^\epsilon$, such that $(f,g)$ is a Morse-Smale pair on $B$. Then, if $\epsilon>0$ is sufficiently small,  $f$ is also a small perturbation of $f_\phi^0$!
The next step is to consider the algebra provided by $f$.
The latter implies that $\HM_*(f,g,B,\co)\cong \HI_*(S,\phi)$, where the chain complex for $f$
is given by
$$
\partial^S_k: C_k(S) \to C_{k-1}(S),
$$
with $\partial^S_k = \partial_k(f,g,B,\co)$ and $C_k(S) = C_k(f,B)$,
and $H_*\bigl(C_*(S),\partial_*^S\bigr) = \HM_*(f,g,B,\co)$.

Let $B_A$ and $B_R$ be a isolating blocks for $A$ and $R$ respectively, then the
Morse-Conley-Floer homologies of $A$ and $R$ are defined by restricting the count of the critical points and connecting orbits of $f$ to $B_A$ and $B_R$. We obtain the chain complexes:
$$
\partial^A_k: C_k(A) \to C_{k-1}(A),\quad{\rm and}\quad \partial^R_k: C_k(R) \to C_{k-1}(R),
$$
with $\partial^A_k = \partial_k(f,g,B_A,\co)$ and $C_k(A) = C_k(f,B_A)$
and similarly $\partial^R_k = \partial_k(f,g,B_R,\co)$ and $C_k(R) = C_k(f,B_R)$.
The homologies satisfy: 
\begin{align*}
H_*\bigl(C_*(A),\partial_*^A\bigr) &= \HM_*(f,g,B_A,\co) \cong \HI_*(A,\phi);\\
H_*\bigl(C_*(R),\partial_*^R\bigr) &= \HM_*(f,g,B_R,\co) \cong \HI_*(R,\phi).
\end{align*}
Because of the properties of $f$ we see that $C_k(S)=C_k(A)\oplus C_k(R)$. The gradient system defined by $f$ only allows connections from $R$ to $A$ and not vice versa, and therefore for the negative $g$-gradient flow $\psi_{(f,g)}$ defined by $(f,g)$, the sets
$$
 A_{(f,g)} = \Inv(B_A,\psi_{(f,g)}) <
R_{(f,g)}  = \Inv(B_R,\psi_{(f,g)})
$$
 is a Morse decomposition for $S_{(f,g)}  = \Inv(B,\psi_{(f,g)} )$.
 This Morse decomposition provides additional information about the boundary operator $\partial^S_*$:
$$
\partial^S_k = \left(\begin{array}{cc}\partial_k^A & \delta_k \\0 & \partial_k^R\end{array}\right):
C_k(A) \oplus C_k(R) \to C_{k-1}(A) \oplus C_{k-1}(R),
$$
where $\delta_k:C_k(R) \to C_{k-1}(A)$ counts connections from $ R_{(f,g)} $ to $A_{(f,g)} $ under $\psi$.
For the sub-complexes $\bigl(C_*(A),\partial_*^A\bigr)$ and $\bigl(C_*(R),\partial_*^R\bigr)$
provide natural inclusions and projections
$$
i_k: C_k(A) \to C_k(S) = C_k(A)\oplus C_k(R), \quad a\mapsto (a,0),
$$
and
$$
j_k: C_k(S) = C_k(A) \oplus C_k(R) \to C_k(R), \quad (a,r)\mapsto r.
$$
The maps $i_*$ and $j_*$ are chain maps. Indeed, for $a\in C_k(A)$, and $r\in C_k(R)$, we have 
$$
i_{k-1}\partial^A_k(a)=(0,\partial_k^Aa)=\partial^S_k(0,a)=\partial^S_ki_k(a),
$$
and
$$
j_{k-1}\partial^S_k(a,r)=j_{k-1}(\partial_k^Aa+\delta_kr,\partial^R_kr)=\partial_k^Rr=\partial_k^Rj_k(a,r).
$$

This implies that the induced homomorphisms
$
i_k^*: H_k\bigl(C_*(A)\bigr) \to H_{k-1}\bigl(C_*(A)\bigr)
$ and
$
j_k^*: H_k\bigl(C_*(R)\bigr) \to H_{k-1}\bigl(C_*(R)\bigr)
$
are well-defined. The maps $i_*$ and $j_*$ define the following short exact sequence
\begin{eqnarray}
\label{eqn:short}
0 \rightarrow  C_k(A) \xrightarrow{i_k} C_k(S)\xrightarrow{j_k} C_k(R) \rightarrow 0.
\end{eqnarray}
Since $i_*$ and $j_*$ are chain maps, this is actually a short exact sequence of chain maps, and by the snake lemma we obtain the following long exact sequence in homology:\footnote{The homologies
are abbreviated as follows: $H_k(A) = H_k\bigl(C_*(A),\partial_*^A\bigr)$,  $H_k(R) = H_k\bigl(C_*(R),\partial_*^R\bigr)$ and
 $H_k\bigl(S\bigr) = H_k\bigl(C_*(S),\partial_*^S\bigr)$.}
\begin{align*}
\cdots \xrightarrow{\delta_{k+1}} H_k(A)\xrightarrow{i_k}H_k\bigl(S\bigr)
\xrightarrow{j_k}H_k\bigl(R\bigr)\xrightarrow{\delta_k}H_{k-1}\bigl(A\bigr)\xrightarrow{i_{k-1}}\cdots
\end{align*}
We recall that the connecting homomorphisms $H_k\bigl(R\bigr)\xrightarrow{\delta_k}H_{k-1}\bigl(A)$
are established as follows. Let $r \in \ker \partial_k^R$, then $\bigl[ (i_k^{-1}\circ\partial_k^S\circ j_k^{-1})(r)\bigr] \in H_{k-1}\bigl(A\bigr)$.
Indeed, $j_k^{-1}(r) = \bigl( C_k(A),r\bigr)$ and
$\partial_k^S\bigl( j_k^{-1}(r)\bigr) = \bigl( \partial_k^A C_k(A) + \delta_k r, \partial_k^R r\bigr)
=  \bigl( \partial_k^A C_k(A) + \delta_k r, 0\bigr)$. Furthermore,
$i_k^{-1} \bigl( \partial_k^S\bigl( j_k^{-1}(r)\bigr) \bigr) = \partial_k^A C_k(A) + \delta_k r
= \delta_k r +\im \partial_k^A\in \ker \partial_{k-1}^A$, since $\partial_{k-1}^A\delta_k r = 
- \delta_{k-1} \partial_k^R r = 0$.
The class of $i_k^{-1} \bigl( \partial_k^S\bigl( j_k^{-1}(r)\bigr) \bigr)$ is a homology class
in $H_{k-1}\bigl(R\bigr)$ and the we write $[\delta_k r] = \delta_k [r]$.

Since $B$ is compact the chain complexes involved are all finite dimensional and terminate at $k=-1$ and $k=\dim M+1$ and therefore
the Poincar\'e polynomials are well-defined. The Poincar\'e polynomials satisfy
$$
P_t\bigl(H_*\bigl(A\bigr) \bigr) + P_t\bigl(H_*\bigl(R\bigr) \bigr) 
- P_t\bigl(H_*\bigl(S\bigr) \bigr) = (1+t)Q_t,
$$
where $Q_t =\sum_{k\in\mZ} (\rank \delta_k) t^k$, which completes the proof.
\end{proof}

By definition $\partial_*^A$ and $\partial_*^R$ are zero maps on homology and therefore induced maps on homology give rise to a map
$$
\Delta_k=\left(\begin{array}{cc}0 & \delta_k \\0 & 0\end{array}\right): 
H_k(A) \oplus H_k\bigl(R\bigr) \to H_{k-1}\bigl(A\bigr) \oplus H_{k-1}\bigl(R\bigr).
$$
Define $\Delta=\bigoplus_{k\in \mZ} \Delta_k$ and 
$$
\Delta: \bigoplus_{k\in \mZ} \Bigl( H_k(A) \oplus H_k\bigl(R\bigr)\Bigr) 
\to \bigoplus_{k\in \mZ} \Bigl( H_k(A) \oplus H_k\bigl(R\bigr)\Bigr),
$$
with the property $\Delta^2=0$ and is called the \emph{connection matrix} for the attractor-repeller
pair $(A,R)$. The above defined complex is a chain complex and the associated homology is isomorphic to $H_*\bigl(S\bigr)$, see \cite{Mischaikow:1995wl}.

For a given Morse decomposition $\sfS = \{S_i\}_{i\in I}$ of $S$ the Morse sets $S_i$ are again isolated invariant sets and therefore their Morse-Conley-Floer homology
$\HI_*(S_i,\phi)$ is well-defined. The associated Poincar\'e polynomials satisfy the 
Morse-Conley relations.

\begin{theorem}
\label{thm:Morse-rel}
Let $\sfS = \{S_i\}_{i\in I}$ be a Morse decomposition for an isolated invariant set $S$.
Then,
\begin{equation}
\label{eqn:MRfor-S}
\sum_{i\in I} P_t(S_i,\phi)  = P_t(S,\phi) +  (1+t) Q_t,
\end{equation}
where $P_t(S,\phi)$ is the Poincar\'e polynomial of $\HI_*(S,\phi)$, and $Q_t$ is a matrix with non-negative coefficients. These relations are called the Morse-Conley relations and generalize the classical Morse relations for gradient flows. 
\end{theorem}

\begin{proof}
The Morse decomposition $\sfS$ is equivalent to a lattice of attractors $\sA$, see \cite{KMV}. We choose a chain 
$$
A_0=\varnothing \subset A_1\subset \cdots \subset A_n=S
$$ 
in $\sA$, such that $S_i = A_i \cap A_{i-1}^* = A_{i-1}|^*_{A_i}$.
Each attractor $A_i \subset S$ is an isolated invariant set for $\phi$ in $B$ (see \cite{Conley-Zehnder-2}) and
for each $A_i$ we have the attractor-repeller pair $A_{i-1} < A_{i-1}|^*_{A_i}=S_i$.
From Lemma \ref{lem:poly-2} we derive the attractor-repeller pair Morse-Conley relations for
each attractor-repeller pair $(A_{i-1},S_i)$ in $A_i$:
$$
P_t(A_{i-1},\phi) + P_t(S_i,\phi)  - P_t(A_i,\phi) = (1+t) Q^i_t.
$$
Summing $i$ from $i=1$ through to $i=n$ we obtain the Morse-Conley relations (\ref{eqn:MRfor-S}), which proves the theorem.
\end{proof}

\begin{remark}
\label{rmk:conn-matrix}
{\em 
In the proof of the Morse relations we use an extension of a Morse decomposition to a totally ordered 
Morse decomposition, which yields a connection matrix. If more refined Morse decompositions are used other connections may be obtained. This information is contained in the $Q_t$ term of the Morse relations. 
}
\end{remark}

\section{Relative homology of blocks}
\label{sec:relblock}

An important property of Morse-Conley-Floer homology is that it can be computed in terms of singular relative homology of
a topological pair of a defining blocks as pointed out in Property (vi) in Section~\ref{sec:intro}.

\begin{theorem}
\label{thm:iso-block}
Let $S$ be an isolating neighborhood for $\phi$ an let $B$ be an isolating block for $S$. Then
\begin{equation}
\HI_*(S,\phi) \cong H_*(B,B_-;\mZ),\quad \forall k\in \mZ,
\end{equation}
where 
$B_-  =   \{ x\in \partial B~|~X(x)~ \hbox{is outward pointing}\}$
and is called the `exit set'.
\end{theorem}

Note that in the case that $\phi$ is the gradient flow of a Morse function, then Theorem \ref{thm:iso-block} recovers the results of Morse homology. Theorem \ref{thm:iso-block} also justifies the terminology Morse-Conley-Floer homology, since the construction uses Morse/Floer homology and
recovers the classical homological Conley index.
The following lemma states that one can choose a metric $g$ such for any Lyapunov function $\f$ the boundary behavior of $-\nabla_g \f$ coincides with $X$. 

Recall Definition\ \ref{defn:block} of an isolating block $B$. Because $B$ is a manifold with piecewise smooth boundary, we need to be careful when we speak of the boundary behavior of $B$. We say $h_+:M\rightarrow \mR$ defines the boundary $B_+$ if 
$$X h_+|_{B_+}>0,\quad dh_+|_{B_+}\not =0, \quad\text{and}\quad h_+|_{B\setminus B_+}>0.$$
Analogously we say that $h_-:M\rightarrow \mR$ defines the boundary $B_-$ if 
$$X h_-|_{B_-}<0,\quad dh_-|_{B_-}\not =0, \quad\text{and}\quad h_-|_{B\setminus B_-}>0.$$
\begin{lemma}
\label{lem:choice-g}
Let $\f:M\to \mR$ be a smooth Lyapunov function for $(S,\phi)$ with the Lyapunov property with respect to an isolating block $B$ (cf.~Proposition \ref{prop:lyap1}). Then there exists  a metric $g$ on $M$ which satisfies the property:
$(-\nabla_g \f) h_- < 0$ for all $x\in B_-$ and  $(-\nabla_g \f) h_+> 0$  for all $x\in B_+$.
\end{lemma}
\begin{proof}
Since $X\not=0$ on $\partial B$, there exists an open $U$ containing $\partial B$, such that $X\not=0$, and $Xf_\phi<0$ on $U$. The span of $X$ defines an one dimensional vector subbundle $E$ of the tangent bundle $TM|_U$ over $U$. There is a complementary subbundle $E^\perp\subset TM|_U$ such that $TM|_U\cong E\oplus E^\perp$. Define the inner product $e_E$ on $E$ by $e_E(X,X)=1$, and let $e_{E^\perp}$ be any inner product on $E^\perp$. Declaring that $E$ and $E^\perp$ are orthogonal defines a metric $e$ on $TM|_U$.  Clearly $e(-\nabla_ef_\phi,X)=-Xf_\phi$, and therefore
$$
-\nabla_e f_\phi=(-Xf_\phi)X+Y,\qquad \text{with}\qquad Y\in E^\perp.
$$
Since $-Xf_\phi>0$ on $U$, we can rescale the metric $e$ to $g=(-Xf_\phi) e_E+\frac{1}{\epsilon} e_{E^\perp}$, for some $\epsilon>0$. It follows that $-\nabla_g f_\phi=X+\epsilon Y$. For the boundary defining functions $h_\pm$ we find
$$
(-\nabla_g f_\phi) h_\pm=Xh_\pm+\epsilon Yh_\pm.
$$
Because $\partial B$ is compact there exists a uniform bound $|Y(h_\pm)(x)|\leq C$ for all $x\in \partial B$ independent of $\epsilon$. For $\epsilon>0$ sufficiently small the sign of $(-\nabla_g f_\phi) h_\pm$ agrees with the sign of $Xh_\pm$ on $\partial B_\pm$. Via a standard partition of unity argument we extend the metric $g$ to $M$, without altering it on $\partial B$, which gives a metric with the desired properties.
\end{proof}

\begin{proof}[Proof of Theorem \ref{thm:iso-block}]
By Lemma \ref{lem:choice-g} we can choose a Riemannian metric $g$ such that $-\nabla_g \f$ has the same boundary behavior as  the vector field $X$.
Using Proposition \ref{prop:morse} we have a small $C^2$-perturbation $f$ of $\f$ such that $(f,g) \in \IMS(\f;B)$, via a constant homotopy in $g$.
Since $f$ is sufficiently close to $\f$ the boundary behavior of $-\nabla_g f$ does not change!
From the definition of the Morse-Conley-Floer homology we have that, for $\cQ = (f,g,B,\co)$,
$$
\HI_*(S,\phi) \cong \HM_*(\cQ).
$$
It remains to compute the Morse homology of the Morse-Conley-Floer quadruple $\cQ$.
Relating the Morse homology to the topological pair $(B,B_-)$ is the same  as in the case when $B=M$, which is described in \cite{Banyaga:2004wk} and \cite{Salamon1}.
The arguments can be followed verbatim and therefore we only provide a sketch of the proof.

The first part of the of the proof starts with the boundary operator $\partial_*(\cQ)$. The latter can be related to the boundary operator
in relative singular homology.
Since the negative $g$-gradient flow $\psi_{(f,g)}$ is Morse-Smale all critical points $x\in \crit(f)\cap B$ are isolated invariant sets for $\psi_{(f,g)}$. Let $B^x$ be an isolating block for $S= \{x\}$. From standard Morse theory and Wazeski's principle it follows that
$$
H_k(B^x,B^x_-;\mZ) \cong C_k(f,B^x) \cong \mZ,\quad \hbox{for~} k=\ind(x),
$$
and $H_k(B^x,B^x_-;\mZ) =0$ for $k\not = \ind(x)$.
For critical points $y,x\in \crit(f)\cap B$, with $\ind(x) = \ind(y)+1 =k$ we define the set
$$
S(x,y) = W_B(x,y) \cup \{x,y\},
$$
which is an isolated invariant set with isolating neighborhood $N$. Let $c$ be such that $f(y)<c<f(x)$. For $T>0$ sufficiently large, and $\epsilon>0$ sufficiently small, define the isolating block
\begin{align*}
B^x=\bigr\{z\in N\,|\,\psi_{(f,g)}(-t,z)\in N,\quad &f(\psi_{(f,g)}(-t,z))\leq f(x)+\epsilon,\\
&\forall\, 0\leq t\leq T,\quad f(z)\geq c\bigl\},
\end{align*}
for $\{x\}$, and the isolating block
\begin{align*}
B^y=\bigr\{z\in N\,|\,\psi_{(f,g)}(t,z)\in N,\quad &f(\psi_{(f,g)}(t,z))\geq f(y)-\epsilon,\\ &\forall\, 0\leq t\leq T,\quad f(z)\leq c\bigl\},
\end{align*}
for $\{y\}$. The exit sets are
\begin{align*}
B_-^x&=\{z\in B^x\,|\,f(z)=c\},\\
B_-^y&=\{z\in B^y\,|\,f(\psi_{(f,g)}(T,z))=f(y)-\epsilon\}.
\end{align*}
Define the sets $B_2^{x,y}=B^x\cup B^y$, $B_1^{x,y}=B^y\cup B_-^x$, and $B_0^{x,y}= B_-^y\cup \cl(B_-^x\setminus B^x)$. The set $B_2^{x,y}$ is an isolating block for $S(x,y)$, the set $\cl(B^{x,y}_2\setminus B^{x,y}_1)$ is an isolating block for $\{x\}$, and the set $\cl(B^{x,y}_1\setminus B^{x,y}_0)$ is an isolating block for $\{y\}$. Via the triple $B^{x,y}_0 \subset B^{x,y}_1 \subset B^{x,y}_2$
we define the operator $\Delta_k: H_k(B^{x,y}_2,B^{x,y}_1) \to H_{k-1}(B^{x,y}_1,B^{x,y}_0)$ by the commutative diagram
$$
\xymatrix{
H_*(B^y,B^y_-)\ar[d]^{\cong}\ar[r]^{\Delta_k}&H_*(B^x,B^x_-)\ar[d]^{\cong}\\
H_*(B^{x,y}_2,B^{x,y}_1)\ar[r]^{\delta_k}&H_*(B^{x,y}_1,B^{x,y}_0).
}
$$
The vertical maps express the homotopy invariance of the Conley index, and the horizontal map is the connecting homomorphism in the long exact sequence of the triple. The homomorphism $\Delta_*$ can be defined on $C_*(f,B)$ directly and 
the analysis in \cite{Banyaga:2004wk} and \cite{Salamon1} shows that $\Delta_* = \partial_*(\cQ)$, which that yields that orbit counting can be expressed in terms of  algebraic topology.

The next step to the apply this to the isolated invariant set $S_{(f,g)}$.
Following the proofs in  \cite{Banyaga:2004wk} and \cite{Salamon1} we construct a special Morse decomposition 
of $S_{(f,g)}$.
Let 
$$
S^{k,\ell} = \bigcup\bigl\{ W_B(y,x)~|~k\le \ind(x)\le \ind(y)\le \ell\bigr\},
$$ 
and since $\psi_{(f,g)}$ is Morse-Smale, these sets are compact isolated invariant sets contained in $S_{(f,g)}$, with
  $S^{k,\ell} = \varnothing$ for $k<\ell$. The sets $\{S^{k,k}\}$ form a Morse decomposition 
  of $S_{(f,g)}$ via
$S^{k,k} \le S^{\ell,\ell}$ if and only if $k \le \ell$.
This yields a filtration of blocks $B_k$
$$
B_- \subset B_0 \subset \cdots \subset B_{m-1} \subset B_m = B,
$$
such that $\Inv\bigl(B_\ell\setminus B_{k-1}\bigr) = S^{\ell,k}$. 
We now use a
modification of  arguments in \cite{Banyaga:2004wk} and \cite{Salamon1}.

As before
$H_k(B_k,B_{k-1}) \cong C_k(f,B)$
and $H_\ell(B_k,B_{k-1}) =0$ for $\ell\not = k$.
Consider the triple $B_- \subset B_{k-1} \subset B_k$ and the associated homology long exact sequence
$$
\to H_{\ell+1}(B_k,B_{k-1}) \to H_\ell(B_{k-1},B_-) \to H_\ell(B_k,B_-) \to H_\ell(B_k,B_{k-1}) \to.
$$
For $\ell \not = k-1,k$ the sequence reduces to
$$
0 \to H_\ell(B_{k-1},B_-) \to H_\ell(B_k,B_-) \to 0,
$$
which shows that $H_\ell(B_{k-1},B_-) \cong H_\ell(B_k,B_-)$ for $\ell \not = k-1,k$.
By a left and right induction argument we obtain that $ H_\ell(B_k,B_-) =0$ for all $\ell>k$ and
$ H_\ell(B_k,B_-) \cong  H_\ell(B,B_-)$ for all $\ell <k$.
If we also use the homology long exact sequences
for the triples $B_-\subset B_{k-2} \subset B_{k-1}$ and $B_{k-2} \subset B_{k-1} \subset B_k$ we obtain the commuting diagram
$$
\xymatrix{
&~&~&0\ar[d]\\
0 \ar[r] &H_k(B_k,B_-) \ar[r] &H_k(B_k,B_{k-1}) \ar[r]^{\delta_k} \ar[rd]^{\delta_k'} &H_{k-1}(B_k,B_-) \ar[d]\\
&~&~&H_{k-1}(B_{k-1},B_{k-2}).
}
$$
The following information can be deduced from the diagram.
The vertical exact sequence we derive that the map $H_{k-1}(B_k,B_-) \to H_{k-1}(B_{k-1},B_{k-2})$ is injective and thus from the commuting triangle we conclude that 
$\ker \delta_k = \ker \delta_k' \cong \ker \partial_k$. From the horizontal exact sequence
we obtain that $H_k(B_k,B_-)\cong \ker \delta_k \cong \ker \partial_k$.
 From  these isomorphisms we obtain the following commuting
diagram of short exact sequences:

$$
\xymatrix{
H_{k+1}(B_{k+1},B_{k})\ar[d]^{\cong}\ar[r]^{\delta_{k+1}}&H_{k}(B_{k},B_-)\ar[d]^{\cong}\ar[r]^{\cong}&H_{k}(B_{k+1},B_-)\ar[d]^{\cong}\ar[r]^{}&0\\
C_{k+1}(f,B)\ar[r]^{\partial_{k+1}}&\ker \partial_k\ar[r]^{\cong}&H_{k}(B,B_-)\ar[r]^{}&0,
}
$$
where we use long exact sequence of the triple $B_-\subset B_k\subset B_{k+1}$ and the fact that $H_k(B_{k+1},B_k) =0$.
The diagram implies that $\HI_k(S,\phi) \cong \ker \partial_k / \im \partial_{k+1} \cong H_k(B,B_-)$, which completes the proof.
\end{proof}

\begin{remark}
\label{rmk:MH}
{\em
Theorem \ref{thm:iso-block} relates Morse-Conley-Floer homology to the singular homology of a pair $(B,B_-)$.
The proof uses the fact that Morse-Conley-Floer homology is well-defined as the Morse homology of a 
Morse-Conley-Floer quadruple. In fact one proves that Morse-Conley-Floer homology is isomorphic to 
Morse homology of manifold pairs as developed in \cite{Schwarz:1993wg}.
To be more precise let $B\subset D$ and $D$ is a manifold with boundary $\partial D = D_+$. Let $E\subset D$ be
a manifold pair such that $B= \cl(D\setminus E)$ and $B_- = B\cap E$.
In order to define $\HM_*(D,E)$ we consider Morse-Smale pairs $(f,g)$ with $-\nabla_g f$ inward pointing on
$\partial D = D_+$ and `pointing into' $E$ on $\partial E$.
Then $\HI_*(S,\phi) \cong \HM_*(D,E) \cong H_*(D,E) \cong H_*(B,B_-)$.
It is worthwhile to develop Morse homologies for blocks $(B,B_-)$.
}
\end{remark}
  
\bibliographystyle{abbrv}
\bibliography{Morse-Conley-Floer}

\end{sloppypar}
\end{document}